\documentclass{article}
\usepackage{epsfig}
\usepackage{amsmath, amssymb, amsthm}	
\usepackage{color}		

\usepackage{xcolor}

\newcommand{\R}{{I\!\!R}}
\newcommand{\N}{{I\!\!N}}

\def\R{{\rm I}\! {\rm R}}

\def\X{{\bf X}}

\newtheorem{theorem}{Theorem}[section]

\newtheorem{remark}{Remark}[section]

\newcommand{\OT}{\mathcal O}

\begin{document}

\pagestyle{headings}

\title{Iterative Splitting Methods for Coulomb Collisions in Plasma Simulations}
\author{J\"urgen Geiser
\thanks{Ruhr University of Bochum, Department of Electrical Engineering and Information Technology, Universit\"atsstrasse 150, D-44801 Bochum, Germany, E-mail: juergen.geiser@ruhr-uni-bochum.de}}

\maketitle

\begin{abstract}
In this paper, we present splitting methods that are based on 
iterative schemes and applied to plasma simulations.
The motivation arose of solving the Coulomb collisions,
which are modeled by nonlinear stochastic differential equations. 
We apply Langevin equations to model the
characteristics of the collisions and we obtain coupled nonlinear 
stochastic differential equations, which are delicate to solve. 
We propose well-known deterministic splitting schemes that can be extended 
to stochastic splitting schemes, by taking into account the 
stochastic behavior. The benefit decomposing the different
equation parts and solve such parts individual is taken
into account in the analysis of the new iterative splitting schemes.
Numerical analysis and application to
various Coulomb collisions in plasma applications are presented.

\end{abstract}



\section{Introduction}

We are motivated to develop fast algorithms to
solve Coulomb collisions in plasma simulations.
Such modeling equations results in characteristics
equations, which are nonlinear stochastic differential
equations with different time-scales.
Based on the nonlinearities and multiscale problems
such differential equations are solved by
higher order stochastic solvers, e.g.,  Milstein scheme,
see \cite{kloeden1992} and \cite{cohen2010}.
Such solvers are direct or non-iterative and have the
drawback in missing relaxations of such nonlinear parts, 
see \cite{geiser_2011} and \cite{kelley95}.

Therefore, we propose new iterative splitting schemes, see \cite{geiser2013},
which allow to obtain higher order accuracy with 
a nonlinear solver effect which is related to the fixpoint scheme,
see \cite{geiser_2016}.

In the paper, we discuss the two directions of solver methods for the nonlinear
stochastics differential equations
\begin{itemize}
\item Direct methods: Euler-Maruyama and Milstein schemes, see \cite{kloeden1992},
\item Indirect methods: Iterative splitting schemes, see \cite{geiser_2011}.
\end{itemize}

From the methodological point of the methods, we have historically two ideas 
for algorithms to solve the Coulomb collisions in particle simulations.
Such methods are based on finite-sized particles, whose characteristics are lying on a grid  (e.g. particle-in-cell (PIC) simulation). Here, we have the 
following methods:
\begin{itemize}
\item Binary algorithm: Particles in a finite cell are 
organized into discrete pairs of interacting particles.
The collision is based on the scattered velocities through an angle whose
statistical variance is dictated by the theory of Coulomb collisions
\cite{nanbu1997} and \cite{taki1977}.
\item Test particle algorithm: The collisions are modeled by defining test
and field particles. The velocity of the test-particle is modeled by 
Langevin equations with drag and diffusion coefficients, influenced by 
the moments of the field-particle velocity distribution,
which are deposited on the space mesh \cite{cohen2006}, \cite{jones1996}, 
\cite{lemons2009}, \cite{mann1997} and \cite{sherlock2008}.
\end{itemize}

The underlying model equation for the particle simulation is the Fokker-Planck equation, which is given as
\begin{eqnarray}
\frac{\partial}{\partial t} f({\bf v}) = - \frac{\partial}{\partial {\bf v}} ({\bf F}_d(v) f({\bf v})) + \frac{1}{2} \frac{ \partial^2}{\partial {\bf v} \partial {\bf v}} ( D(v) f(v)) ,
\end{eqnarray}
where ${\bf F}_d = \langle \Delta {\bf v}/ \Delta t \rangle$ and $D = \langle \Delta {\bf v} {\bf v}/ \Delta t \rangle$, and $\langle \cdot \rangle$ are the expected values, which are given as ensemble-averaged drag and diffusion coefficients (see the derivation in \cite{dimits2010} and \cite{dimits2013}). 

Based on the Fokker-Planck equation, we can shift to the 
velocity dependent Langevin equation with an embedded collision 
operator, which is related to an explicit derivation, e.g., \cite{eric1994}.

For a test particle with velocity $v$ we have the following equation:
\begin{eqnarray}
&& dv(t) = F_d(v) dt + \sqrt{2D_v(v)} dW_v(t) , \\
&& d\mu(t) = - 2 D_a(v) \mu dt + \sqrt{2D_a(v) (1 - \mu^2)} dW_{\mu}(t) , \\
&& d\phi(t) = \sqrt{\frac{2D_a(v)}{(1 - \mu^2)}} dW_{\phi}(t) , \\
&& v_0 = 1.0 , \mu(0) = 0, \phi(0) = 1.0 ,
\end{eqnarray}
where the coordinates $(v, \mu = \cos(\theta), \phi)$
are the underlying spherical coordinates given as $(v, \theta, \phi)$ of the test particle. 
$F_d$ is an ensemble-averaged drag, and $D_v$ and $D_a$ are the diffusion
coefficients. Furthermore, $W_v$, $W_{\mu}$ and $W_{\phi}$ are independent 
of the Wiener processes and $v_0$, $\mu_0$ and $\phi_0$ are the initial-conditions.

The paper is outlined as following. In the Section \ref{oper}, we discuss the 
iterative splitting method for the stochastic differential equations and the 
convergence analysis. The numerical algorithms of the direct and
indirect methods are presented in Section \ref{algorithms}.
 The numerical results are discussed in Section \ref{numer} and we conclude our results in Section \ref{conclu}

\section[Iterative Splitting Method]{Iterative Splitting Method for Stochastic Ordinary Differential Equations}
\label{oper}

The following algorithm is based on the iteration with
a fixed-splitting discretization step-size $\tau$. For the
time-interval $[t^n,t^{n+1}]$, we solve the following sub-problems
consecutively for $i=1, 3, \dots 2m+1$, (cf. \cite{gei_2009_5}):
\begin{eqnarray}
 && d c_i(t) = A
c_i(t) dt \; + \; B c_{i-1} dW_t(t), \;
\mbox{with} \; \; c_i(t^n) = c^{n} \label{kap3_iter_1} \\
&& \mbox{and} \; c_{i}(t^n) = c^n \; , \; c_{0} = 0.0 , \nonumber
\\\label{kap3_iter_2}
&& d c_{i+1}(t) = A c_i(t) \; dt  + \; B c_{i+1}(t) \; dW_t, \; \\
&& \mbox{with} \; \; c_{i+1}(t^n) = c^{n}\; , \nonumber
\end{eqnarray}
 where $c^n$ is the known split approximation at the 
time-level $t=t^{n}$. The split approximation at the time-level $t=t^{n+1}$ 
is defined as $c^{n+1}=c_{2m+2}(t^{n+1})$. Furthermore, $W$ is a Wiener process, see \cite{kloeden1992}.

We can rewrite this into the form of the following ordinary differential equation (ODE):
\begin{eqnarray}
 && \frac{\partial c_i(t)}{\partial t} =  A c_i(t) \; + \; B c_{i-1} \dot{W}_t , \;
\mbox{with} \; \; c_i(t^n) = c^{n} \label{kap3_iter_1} \\
&& \mbox{and} \; c_{i}(t^n) = c^n \; , \; c_{0} = 0.0 , \nonumber
\\\label{kap3_iter_2}
&& \frac{\partial c_{i+1}(t)}{\partial t} = A c_i(t) \;  + \; B c_{i+1}(t) \; \dot{W}_t, \; \\
&& \mbox{with} \; \; c_{i+1}(t^n) = c^{n}\; , \nonumber
\end{eqnarray}
where $ \dot{W}_t = \frac{d W_t}{d t}$.

We present the results of the consistency of our
iterative method extended to stochastic operators, see \cite{geiser2013}. 
For simplicity, we assume the system of operators are
generators of a $C_0$-semigroup based on their underlying operator norms.

\begin{theorem}
Let us consider the abstract Cauchy problem in a Banach space \X
\begin{equation}
\begin{array}{c}
{\displaystyle \partial_t c(x, t) = A c(x, t) + B c(x, t) \dot{W}_t },  \;  x \in \Omega \times [0, T] , \\
\noalign{\vskip 1ex} {\displaystyle c(x,0)=c_0(x) } \; x \in \Omega , \\
\noalign{\vskip 1ex} {\displaystyle c(x,t)=c_1(x, t) } \; x \in \partial \Omega \times [0, T], \\

\end{array} \label{eq:ACP}
\end{equation}

\noindent where $A,B: \! {\X} \rightarrow {\X} $ are given linear operators 
that are generators of the $C_0$-semigroup and $c_0 \in \X$ is a given element.

The iterative operator splitting method has the following splitting error:
\begin{eqnarray}
&& || (S_i - \exp(A \tau + B W)|| \le C \tau^{\frac{i+1}{2}},
\end{eqnarray}
where $S_i$ is the approximated solution for the i-th iterative step
and $C$ is a constant that can be chosen uniformly on bounded time
intervals.

\end{theorem}

\begin{proof}

The iterative steps are given in the following.

\begin{itemize}
\item
For the first iterations, we have:
\begin{equation}
\begin{array}{c}
 \partial_t c_1(t) = A c_1(t) + B \dot{W}_t c_0 , \quad t \in (t^n,t^{n+1}],
\end{array} \label{eq:err1}
\end{equation}
where we have the solution given as:
\begin{eqnarray}
 c_1(t) & = & \exp(A t) c(t^n) + \int_0^{t} \exp(A (t-s)) B  \dot{W}_t c(t^n) ds , \quad t \in (t^n,t^{n+1}], \\
& = & \exp(A t) c(t^n) \nonumber \\
&& + (I + A t) \int_0^{t} \exp(- A s) B \exp(B W_s) dW_s + \OT(t^{3/2}), \\
& = & \exp(A t) c(t^n) \nonumber \\
&& + (I + A t) \int_0^{t} (- A B s + B + B B^t W_s) dW_s + \OT(t^{3/2}), \\
& = & \exp(A t) c(t^n)  \nonumber \\
&& + (I + A t) \left( B W_t - A B  t W_t + \frac{1}{2} B B^t W_t^2 -  \frac{1}{2} B B^t t \right) + \OT(t^{3/2}), \nonumber \\ 
& = & ( I + A t + B W_t + \frac{1}{2} B B^t W_t^2 -  \frac{1}{2} B B^t t ) c(t^n) + \OT(t^{3/2}), 
 \label{eq:err1}
\end{eqnarray}
where $c_0(t) = \exp(B W_t) c(t^n)$. \\

Then, the consistency of the first iterative step is given 
in the following.

For $e_1$, we have:
\begin{eqnarray}
 c_1(t) & = & (I + A t + B W_t + \frac{1}{2} B B^t W_t^2 -  \frac{1}{2} B B^t t) c(t^n) + \OT(t^{3/2})) , \\
 c(t) & = & \exp((A-B B^t /2) t + B W_t)  c(t^n) \nonumber \\
     & = & (I + A t + B W_t + \frac{1}{2} B B^t W_t^2 -  \frac{1}{2} B B^t t) c(t^n) + \OT(t^{3/2})) ,
\end{eqnarray}

We obtain:
\begin{eqnarray}
&& || e_1 || = || c - c_1 || \le || \OT(t^{3/2}) . \nonumber
\end{eqnarray}

\item For the second iteration, we have:
\begin{equation}
\begin{array}{c}
 \partial_t c_2(t) = A c_2(t) + B \dot{W}_t c_1 , \quad t \in (t^n,t^{n+1}],
\end{array} \label{eq:err1}
\end{equation}
where we have the solution given as:
\begin{eqnarray}
 c_2(t) & = & \exp(A t) c(t^n) + \int_0^{t} \exp(A (t-s)) B  \dot{W}_t c_1(s) ds , \quad t \in (t^n,t^{n+1}], \\
& = & (I + A t) c(t^n) \nonumber \\
&& + (I + A t) \int_0^{t} (I - A s) B (I + A s) dW_s  \nonumber \\
&& +  (I + A t) \int_0^{t} (I - A s) B (I + A s) \int_0^s (I - A s_1) B (I + A s_1) dW_{s_1} \; dW_s  \nonumber \\
&& + \OT(t^{2}), \\
& = & (I + A t) c(t^n) \nonumber \\
&& + (I + A t) \int_0^{t} (B W_t - t A B W_t + \frac{1}{2} A B W_t t + B A t W_t s- \frac{1}{2} B A W_t t  \nonumber \\
&& +  (I + A t) (B^2 \frac{1}{2} W_t^2 - \frac{1}{2} B^2 t ) + \OT(t^{2}), 
 \label{eq:err1}
\end{eqnarray}
and we apply the second order
accurate integration of $\int_0^t AB W_s ds = \frac{1}{2} A B t W_t$. \\

Then, the consistency of the second iterative step is given 
in the following.

For $e_2$, we have:
\begin{eqnarray}
 c_2(t) & = & (I + A t + B W_t + \frac{1}{2} B B^t W_t^2 -  \frac{1}{2} B B^t t \\
&& \frac{1}{2}BA t W_t + \frac{1}{2}A t W_t ) c(t^n) c(t^n) + \OT(t^{2})) , \\
 c(t) & = & \exp((A-B B^t /2) t + B W_t)  c(t^n) \nonumber \\
     & = & (I + A t + B W_t + \frac{1}{2} B B^t W_t^2 -  \frac{1}{2} B B^t t\\
&& \frac{1}{2}BA t W_t + \frac{1}{2}A t W_t - \frac{1}{2}B^3 t W_t  ) c(t^n)+ \OT(t^{2})) ,
\end{eqnarray}
where we assume  $\frac{1}{2}B^3 t W_t \approx 0$.

We obtain:
\begin{eqnarray}
&& || e_2 || = || c - c_2 || \le \OT(t^{2}) . \nonumber
\end{eqnarray}

With the next iterative step $i=3$, we gain $\frac{1}{2}B^3 t W_t$
and we obtain a full second order scheme.
\end{itemize}

\end{proof}

\begin{remark}
We obtain a higher order scheme for the iterative splitting method. For 
each iterative step, we obtain additional a half order accuracy,
means $\OT(t^{1+\frac{1}{2} i})$, where $i=1,2,3, \ldots$, is the number of iterative
steps. 
\end{remark}

\section[Numerical Algorithms]{Numerical Algorithms for the Nonlinear Stochastic Ordinary Differential Equations}
\label{algorithms}

In the following, we deal with the different numerical algorithms
to solve the nonlinear stochastics differential equations.

We deal with the underlying nonlinear stochastics differential
equation, which is given as:
\begin{eqnarray}
\label{non_stoch}
&& d X = A(X) X dt + B(X) X dW , 
\end{eqnarray}
where $A, B$ are matrices in $\R^{m \times m}$ with $m$ is the number
of unknown. Further, the components of the matrices are dependent
of the solution $X$. Further, the initial values are given as $X_{t_0}= X_0$
and $W$ is Wiener process, see \cite{kloeden1992}.

In the following, we deal with the direct and indirect algorithms,
which are implemented in the numerical experiments.
The direct methods are numerical standard methods, which are
used in the numerical approximation of stochastic 
differential equations. They are simply to implement and
obtain direct the numerical solutions (one-step methods), while
they have their drawback in the resolution of the nonlinear
solutions, while the linearization is given by the 
time-step. Instead the indirect methods are iterative 
solvers and obtain higher order resolutions with additional
iterative cycles (multi-step methods), such that they allow
to resolve the nonlinear solution in the time-step approach,
see \cite{geiser_2011} and \cite{geiser_2016}.

\subsection{Direct Algorithms}

In the following, the standard numerical schemes for
solving the nonlinear stochastics equation (\ref{non_stoch}) are given:
\begin{itemize}
\item Euler-Maruyama scheme is given as:
\begin{eqnarray}
&& X_{n+1} = X_n + A(X_n) X_n \Delta t +  B(X_n) X_n (W_{t_{n+1}} - W_{t_n}),
\end{eqnarray}
for $n = 0, 1, \ldots, N-1$, $X_{0}= X_{t_0}$, and $\Delta t = t_{n+1} - t_n$
is the time-step. Further, $\Delta W = W_{t_{n+1}} - W_{t_n}$ is the stochastic
step based on a Wiener process, see \cite{kloeden1992}. 
\item Milstein scheme is given as:
\begin{eqnarray}
 X_{n+1} && = X_n + A(X_n) X_n \Delta t +  B(X_n) X_n (\Delta W) \nonumber  \\
     && + \frac{1}{2}  B(X_n) X_n \frac{\partial B(X) X}{\partial X}|_{X_n} ((\Delta W)^2 - \Delta^2),
\end{eqnarray}
for $n = 0, 1, \ldots, N-1$, $X_{0}= X_{t_0}$ and $\Delta t = t_{n+1} - t_n$
is the time-step. Further, $\Delta W = W_{t_{n+1}} - W_{t_n}$ is the stochastic
step based on a Wiener process, see \cite{kloeden1992}. 

\item A-B Splitting method, see the ideas in \cite{ninomya2008}, which is given as:

We assume, that we have an approximated solution of the nonlinear stochastic
differential equation (\ref{non_stoch}).
We assume the following fixpoint of the operators, which are
given as $A(X^{n}) \rightarrow \tilde{A}$ and $B(X^{n}) \rightarrow \tilde{B}$ for $n \rightarrow \infty$, where $X^n = X(t^{n})$.

Then, we obtain:
\begin{eqnarray}
&& X_{n+1} = X_0 \exp((\tilde{A} - \frac{\tilde{B} \tilde{B}^t}{2}) (n+1) \Delta t + \tilde{B} \sum_{i = 1}^{n+1} \Delta W_{i-1} ),
\end{eqnarray}
where we assume $W = \{ W_t, t \ge 0 \}$ and 
$\Delta W_{i-1} = W_{i-1}(t_{n+1}) - W_{i-1}(t_{n})$, where $\Delta t = t_{n+1} - t_n$ and we assume an equidistant grid.

Then, we obtain the following A-B splitting approach:
\begin{eqnarray}
&& \tilde{X}_{n} = X_{n-1} \exp((\tilde{A} - \frac{\tilde{B} \tilde{B}^t}{2}) \Delta t) ,\\ 
&& X_{n} = \tilde{X}_{n} \exp(\tilde{B} \Delta W) ,
\end{eqnarray}
for $n = 0, 1, \ldots, N-1$, $X_{0}= X_{t_0}$.

\end{itemize}

\begin{remark}
The direct methods are fast to implement and obtain lower order results.
The numerical scheme have the following accuracy: $\OT(t^{\frac{1}{2}})$ for the
Euler-Maruyama scheme),  $\OT(t)$ for the Milstein scheme and 
$\OT(t^{\frac{1}{2}})$ for the AB-splitting scheme for large n $n \rightarrow \infty$. Here, the approach to higher order schemes are delicate, see \cite{kloeden1992}.
\end{remark}

\subsection{Indirect Algorithms (iterative splitting)}

In the following, we discuss the iterative splitting methods
for the nonlinear stochastic equation (\ref{non_stoch}).

\begin{itemize}
\item
First iterative step
\begin{eqnarray}
 X_{1,n}(t) = \phi_1(t) X_{n-1} , 
\label{iter_1_1}
\end{eqnarray}
where $\phi_1(t) = \exp(A(X_{n-1}) \Delta t)$ is the first order
approximation of the non-linear Magnus-expansion. \\

\item
Second iterative step
\begin{eqnarray}
&&  X_{2,n}(t) = X_{1,n}(t) \nonumber \\
&& + X_{1,n}(t) [B(X_{n-1}), \int_0^t  \exp(A(X_{n-1})s) dW_s ,  \quad t \in (t^n,t^{n+1}] , \nonumber \\
&& X_{2,n}(t) = X_{1,n}(t) + X_{1,n}(t) [B(X_{n-1}) , C_1(t)] ,  \quad t \in (t^n,t^{n+1}] , \nonumber \\
&& X_{2,n}(t) = X_{1,n}(t) + X_{1,n}(t) C_2(t) ,  \quad t \in (t^n,t^{n+1}] ,  
 \label{eq:err1}
\end{eqnarray}
where $C_1(t) = \int_0^t \exp(A(X_{n-1}) s) dW_s$ $ \Delta W_i = (W_{t_{i+1}} - W_{t_i})$,
for $n = 0, 1, \ldots, N-1$, $X_{0}= X_{t_0}$.

The stochastic integral is computed as Stratonovich integral:
\begin{eqnarray}
\label{int_1_1}
 && C_1(\tilde{t}) = \int_0^{\tilde{t}}  \exp(A(X_{n-1}) s) dW_s \\
&& = \sum_{j=0}^{N-1}  \exp(A(X_{n-1}) (\frac{t_{j} + t_{j+1}}{2})) \;  (W(t_{j+1}) - W(t_j)) , \nonumber \\
&& \Delta t = \tilde{t} /N, t_{j} = \Delta t + t_{j-1}, t_0 = 0 ,
\end{eqnarray}
and the commutator $[\cdot, \cdot]$ is computed as:
\begin{eqnarray}
\label{int_1_2_1}
C_2(t) = [B(X_{n-1}), C_1(t)] = B(X_{n-1}) C_1(t) - C_1(t) B(X_{n-1}) ,
\end{eqnarray}
which is based on the different random variables of $C_1(t)$.
Additionally, in the
scalar case, the commutator is not equal to zero.

\item
Third iterative step
\begin{eqnarray}
 && X_{3,n}(t) =  X_{2,n}(t) + X_{1,n}(t) \; \int_0^t [B(X_{n-1}), \exp(s A(X_{n-1}))] \cdot \\
&& \cdot [B(X_{n-1}), \int_0^{s} \exp(A(X_{n-1}) s_1) ds_1] \;  ds , \nonumber \\
&& X_{3,n}(t) =  X_{2,n}(t) + X_{1,n}(t) \; \int_0^t [B(X_{n-1}), \exp(s A(X_{n-1}))] \;  C_2(s) \;  ds , \nonumber \\
&& X_{3,n}(t) =  X_{2,n}(t) + X_{1,n}(t) \;  C_3(t) , 
 \label{eq:err1}
\end{eqnarray}
where $ \Delta W_i = (W_{t_{i+1}} - W_{t_i})$,
for $n = 0, 1, \ldots, N-1$, $X_{0}= X_{t_0}$.

The operator $C_3(t)$ is computed as:
\begin{eqnarray}
\label{equ2}
  C_3(t) && = \sum_{j=0}^{N-1}  \left( B(X_{n-1}) \exp(A(X_{n-1}) \; (\frac{t_{j} + t_{j+1}}{2})) \; C_2(\frac{t_{j} + t_{j+1}}{2}) \right. \\
&& \left. - \exp(A(X_{n-1}) \; (\frac{t_{j} + t_{j+1}}{2})) B \; C_2(\frac{t_{j} + t_{j+1}}{2}) \right) , \nonumber \\
&& \Delta t = t /N, t_{j} = \Delta t + t_{j-1},  t_0 = 0 ,
\end{eqnarray}
where $C_2(\frac{t_{j} + t_{j+1}}{2})$ is computed with (\ref{int_1_2_1}), for each $\tilde{t} = \frac{t_{j} + t_{j+1}}{2}$, $j=0, \ldots, N-1$.

\end{itemize}

\begin{remark}
The indirect methods are based on the iterative approaches
related to fixpoint-schemes and obtained higher order accuracy: $\OT(t^{i+\frac{1}{2}})$,
where $i$ is the number of iterative steps.
Based on their recursive behavior numerical approaches in previous 
iterative steps can be used. Such a clever combination of the previous computed
iterative cycles allows to obtain fast iterative methods, see \cite{geiser_2011_2}
and \cite{geiser_2011_3}.
\end{remark}

\section{Numerical Examples}
\label{numer}

In the following numerical examples, we verify the theoretical results and
the benefits of the novel iterative solvers for the
stochastic differential equations.

We deal with the following examples and discuss the methodological sense of 
the different schemes:
\begin{itemize}
\item  Scalar benchmark problem (scalar multiplicative noise): 
The stochastic differential equations are based on $m \times m$ operator matrices,
while we have a scalar stochastic term. For such a benchmark examples, we can 
detailed analyze the benefit of the iterative scheme, which is related to the
higher order approach.
\item Vectorial benchmark problems (vectorial multiplicative noise): 
The stochastic differential equations are based on $m \times m$ operator matrices
and we have vectorial stochastic terms. Such vectorial examples need additional, so
called outer-diagonal entries for the standard scheme, see \cite{kloeden1992} and \cite{dimits2013}. 
For the iterative schemes, we have also an extension to resolves such multiple integrals
based on the vectorial stochastics, see \cite{tocino2009}.
Here, we can analyze the benefit of the additional terms and the higher accuracy of the 
novel methods. Further, we also extend such problems to larger operator matrices to see the
computational amount of the different schemes.
\item Real-life problem (Coulomb test-particle):
Here, we test a system of nonlinear stochastic differential equations with vectorial stochastic terms.
Such examples are delicate to solve and we apply the different standard and novel schemes.
For such problems, we see the benefit of the iterative splitting methods, which combine
the linear and nonlinear solvers. We relax the solution based on the iterative
approach and obtain much more accurate results. 
\end{itemize}

\subsection{Scalar multiplicative noise}

In the following, we deal with a simple chemical reaction
model, while the reaction part is influenced via stochastic noise.

We deal first with an ordinary differential equation and separate 
the complex operator into two simpler operators: 
the $m \times m$ ordinary differential equation 
system given as:
\begin{eqnarray}
\label{num_7}
&& d {\bf y}(t) = A {\bf y}(t) +  P {\bf y}(t) \; d W (t) , \\
&& A = \left( \begin{array}{c c c c}
 - \lambda_{1,1} & \lambda_{2,1}  & \dots & \lambda_{1,10}   \\
    \lambda_{2,1} & -  \lambda_{2,2} &  \dots  & \lambda_{2,10}  \\
\vdots \\
   \lambda_{10,1} &  \lambda_{10,2} & \dots & - \lambda_{10,10} \\
\end{array} \right) =  \left( \begin{array}{c c c c}
 - 1 & 0  & \dots & 0   \\
   0.1 & - 1 &  \dots  & 0  \\
\vdots \\
   0.1 & 0.1 & \dots & - 1 \\
\end{array} \right) \nonumber \\
&& P = \left( \begin{array}{c c c c}
 \sigma_{1,1} &  \sigma_{1,2} & \dots &  \sigma_{1,10}    \\
  \sigma_{2,1} &  \sigma_{2,2} &  \dots  &  \sigma_{2,10}  \\
\vdots \\
  \sigma_{10,1}  &  \sigma_{10,2} & \dots &  \sigma_{10,10} \\
\end{array} \right)  = \left( \begin{array}{c c c c}
 0.01 &  0  & \dots & 0    \\
 0.005 & 0.01 &  \dots  & 0   \\
\vdots \\
  0.005  & 0.005 & \dots & 0.01  \\
\end{array} \right) \nonumber \\
 &&  dW(t) = d W_1(t) \mbox{(stochastic scalar)} , \nonumber \\
&&  y(0) = (1, \ldots, 1)^t \mbox{(initial conditions)} \nonumber ,
\end{eqnarray}
where $\lambda_{11} \ldots \lambda_{10, 10} \in \R^+$ are the decay
factors and $\sigma_{11}, \ldots,  \sigma_{10,10} \in \R^+$ are the 
parameters of the perturbations. We deal with non-commutation matrices 
$[A, P] = A P - P A$ as given with the tridiagonal matrices
in the experiment.

We have the time interval $t \in [0,T]$ and $m \in \N$.

We apply the following numerical schemes:
\begin{itemize}
\item
The application of the standard Euler-Maruyama scheme is given as:
\begin{eqnarray}
&& y_{n+1} = y_n + A y_n \Delta t +  P y_n \Delta W ,
\end{eqnarray}
for $n = 0, 1, \ldots, N-1$, $y_{0}= y_{t_0}$, $\Delta t = t_{n+1} - t_n$, $\Delta W = W_{t_{n+1}} - W_{t_n} = \sqrt{\Delta t} N(0,1)$, where $N(0,1) = rand$ is a normally distributed random variable. 
\item
 Milstein scheme is given as:
\begin{eqnarray}
 y_{n+1} && = y_n +  A y_n \Delta t +  P y_n (\Delta W) \nonumber  \\
     && + \frac{1}{2} P P^t y_n ((\Delta W)^2 - \Delta t), 
\end{eqnarray}
for $n = 0, 1, \ldots, N-1$, $y_{0}= y_{t_0}$.
\item Recursive Splitting scheme is given as:
\begin{eqnarray}
\label{recursive_1}
\tilde{y}_{n+1} && = \exp((A - \frac{P P^t }{2}) \Delta t) y_n ,  \\
\label{recursive_2}
 y_{n+1}   && =  \exp(P \Delta W) \tilde{y}_{n+1} , 
\end{eqnarray}
for $n = 0, 1, \ldots, N-1$, $y_{0}= y_{t_0}$.

\item Summative Splitting scheme is given as:
\begin{eqnarray}
\label{summative_1}
\tilde{y}_{n+1} && = \exp((A - \frac{P P^t }{2}) \Delta t) y_n ,  \\
\label{summative_2}
 y_{n+1}   && =  \exp(P \frac{1}{\sqrt{\tilde{N}}}  \sum_{j=1}^{\tilde{N}} \Delta W_j) \tilde{y}_{n+1} , 
\end{eqnarray}
$\Delta W_j = (W(\tilde{t}_{j+1}) - W(\tilde{t}_j)) = \sqrt{\delta t} N(0,1)$, where $N(0,1) = rand$ is a normally distributed random variable. Further the intermediate time-steps are given as $\delta t = \Delta t / \tilde{N}$, $\tilde{t}_{j+1} = \delta t + \tilde{t}_{j}$, $\tilde{t}_1 = t_n$ and the time-intervals are given as $n = 0, 1, \ldots, N-1$, $y_{0}= y_{t_0}$.

\item Iterative splitting scheme:

{\bf Version 1: 2 iterative steps}

Second iterative step:
\begin{eqnarray}
&& X_{2,n}(t) = X_{1,n}(t) + X_{1,n}(t) C_2(t) ,  \quad t \in (t^n,t^{n+1}] ,  
 \label{eq:err1}
\end{eqnarray}
where the commutator is given as:
\begin{eqnarray}
\label{int_1_2}
C_2(t) = [P, C_1(t)] = P C_1(t) - C_1(t) P ,
\end{eqnarray}
where $C_1(t) = \int_0^t \exp(A s) dW_s$

The stochastic integral is computed as a Stratonovich integral:
\begin{eqnarray}
\label{int_1_1}
 && C_1(\tilde{t}) = \int_0^{\tilde{t}}  \exp(A s) dW_s \\
&& = \sum_{j=0}^{N-1}  \exp(A(\frac{t_{j} + t_{j+1}}{2})) \;  (W(t_{j+1}) - W(t_j)) , \nonumber \\
&& \Delta t = \tilde{t} /N, t_{j} = \Delta t + t_{j-1}, t_0 = 0 ,
\end{eqnarray}
where $ \Delta W_i = (W_{t_{i+1}} - W_{t_i})$,
for $n = 0, 1, \ldots, N-1$, $X_{0}= X_{t_0}$.

{\bf Version 2: 3 iterative steps}
\begin{eqnarray}
\label{iterative_3}
&& X_{3,n}(t) =  X_{2,n}(t) + X_{1,n}(t) \;  C_3(t) , 
 \label{eq:err1}
\end{eqnarray}
where $ \Delta W_i = (W_{t_{i+1}} - W_{t_i})$,
for $n = 0, 1, \ldots, N-1$, $X_{0}= X_{t_0}$.

The operator $C_3(t)$ is computed as:
\begin{eqnarray}
\label{equ2}
  C_3(t) && = \sum_{j=0}^{N-1}  \left( B \exp(A(\frac{t_{j} + t_{j+1}}{2})) \; C_2(\frac{t_{j} + t_{j+1}}{2}) \right. \\
&& \left. - \exp(A(\frac{t_{j} + t_{j+1}}{2})) B \; C_2(\frac{t_{j} + t_{j+1}}{2})  \;   (W(t_{j+1}) - W(t_j)) \right) , \nonumber \\
&& \Delta t = t /N, t_{j} = \Delta t + t_{j-1},  t_0 = 0 ,
\end{eqnarray}
where $C_2(\frac{t_{j} + t_{j+1}}{2})$ is computed with (\ref{int_1_2}), for each $\tilde{t} = \frac{t_{j} + t_{j+1}}{2}$, $j=0, \ldots, N-1$.

\end{itemize}

We compare the following schemes:
\begin{itemize}
\item First order (or strong convergence $\OT(t^{1/2})$)
\begin{itemize}
\item EM (Euler-Maruyama): explicit first order Runge-Kutta scheme, see \cite{kloeden1992}.
\item rS (recursive Splitting): modified Lie-Trotter splitting scheme for the
  stochastic term, see Equation (\ref{recursive_1})-(\ref{recursive_2}) and \cite{geiser2013_1} and \cite{geiser2013}.
\item sS (summative Splitting): modified Lie-Trotter splitting scheme with improved computation of the stochastic term, see Equation (\ref{summative_1})-(\ref{summative_2}) and \cite{geiser2013_1} and \cite{geiser2013}.
 \end{itemize}
\item Second order (or strong convergence $\OT(t^{1})$)
\begin{itemize}
\item Mil (Milstein): explicit second order Runge-Kutta scheme, see \cite{kloeden1992}.
\item NV (Niomiya-Victori Splitting): modified Strang-Splitting scheme for the stochastic terms, see \cite{ninomya2008} and \cite{ninomya2009}.
\item iterative splitting ($i=2$)
 \end{itemize}
\item Third order (or strong convergence $\OT(t^{3/2})$)
\begin{itemize}
\item iterative splitting ($i=3$): modified iterative splitting scheme for the stochastic terms, see Equation (\ref{iterative_3}) and \cite{geiser2013}.
 \end{itemize}
\end{itemize}

In the following, we present the results of the lower order
schemes in Figure \ref{lower_real}.
\begin{figure}[ht]
\begin{center}  
\includegraphics[width=5.0cm,angle=-0]{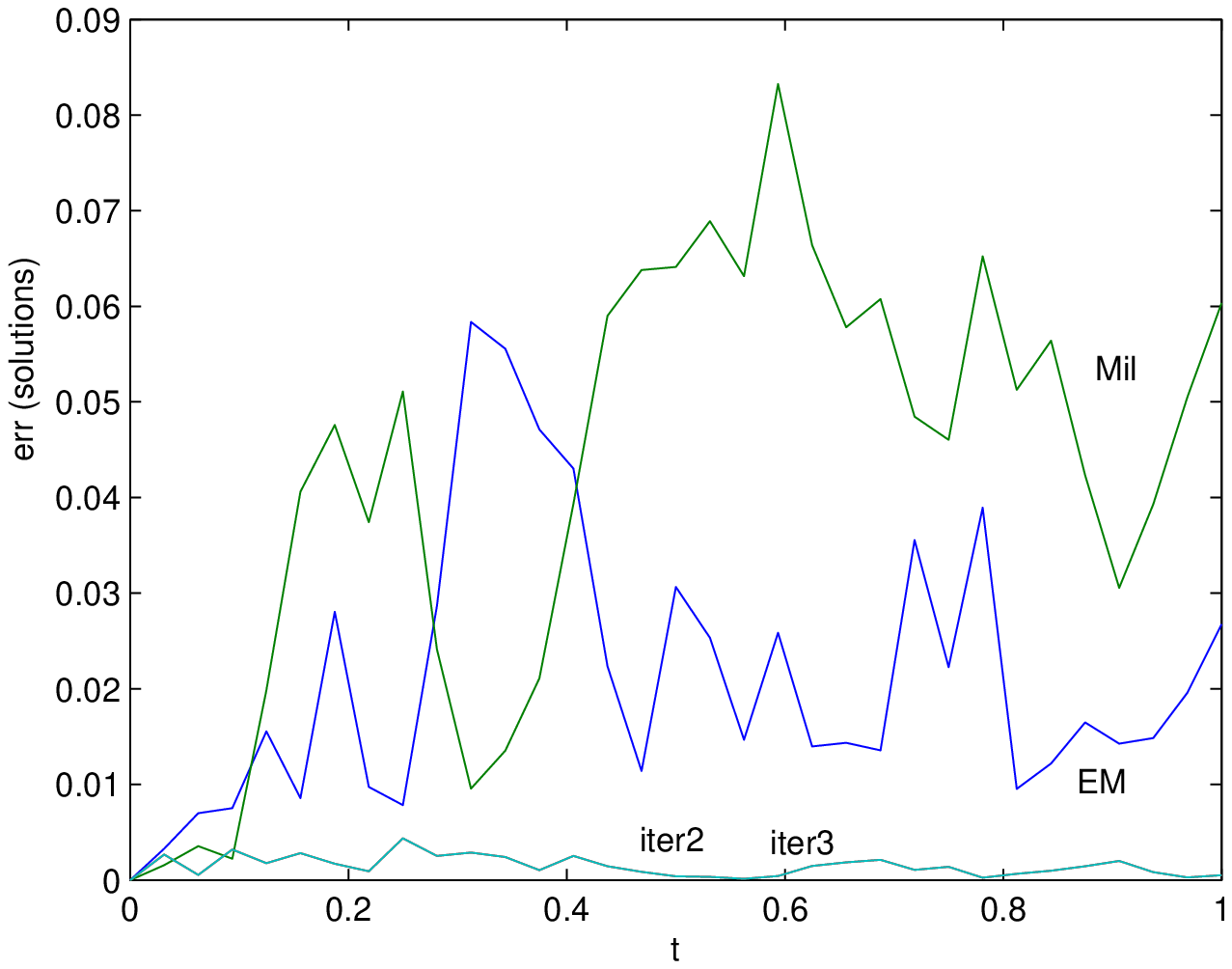}
\includegraphics[width=5.0cm,angle=-0]{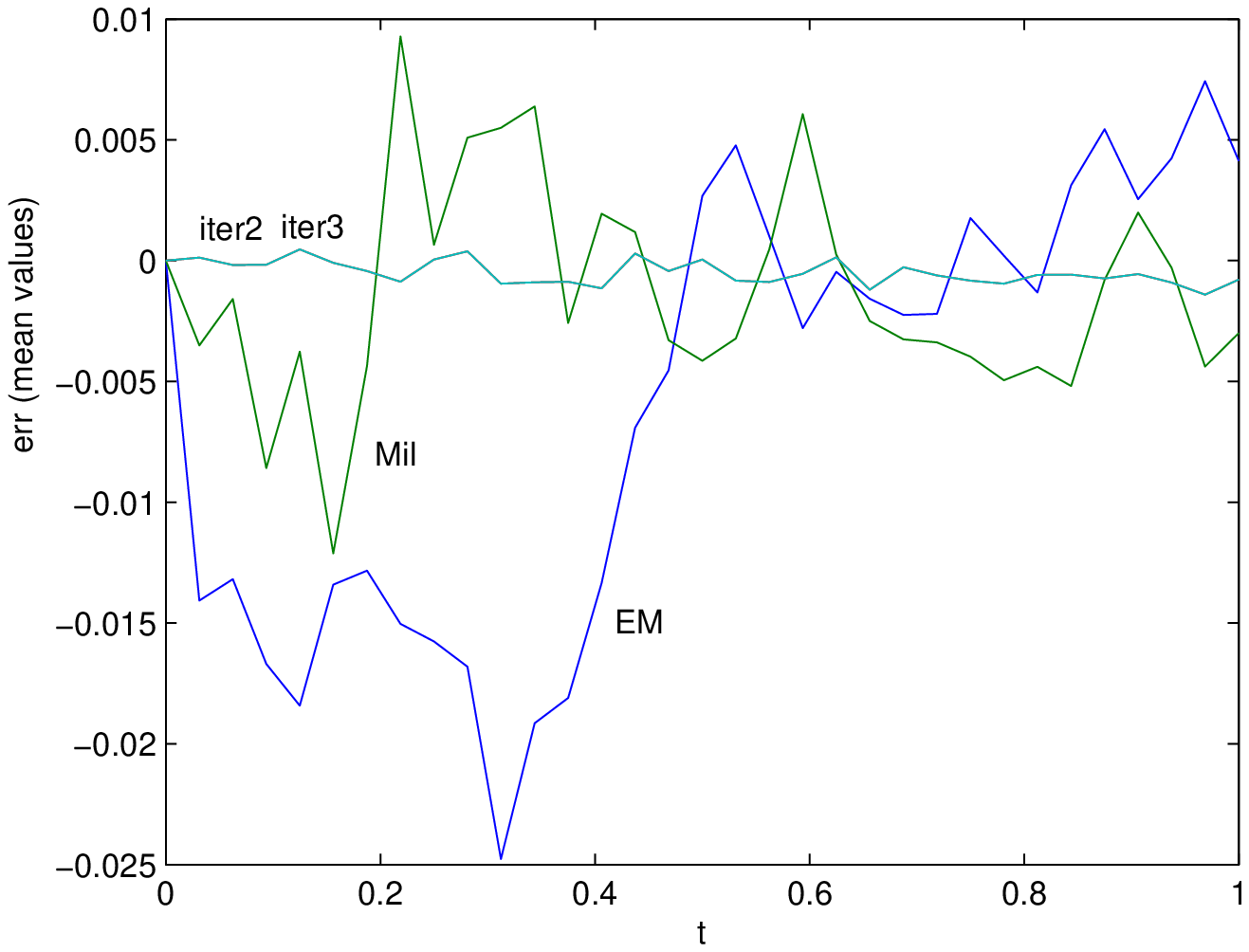} 
\includegraphics[width=5.0cm,angle=-0]{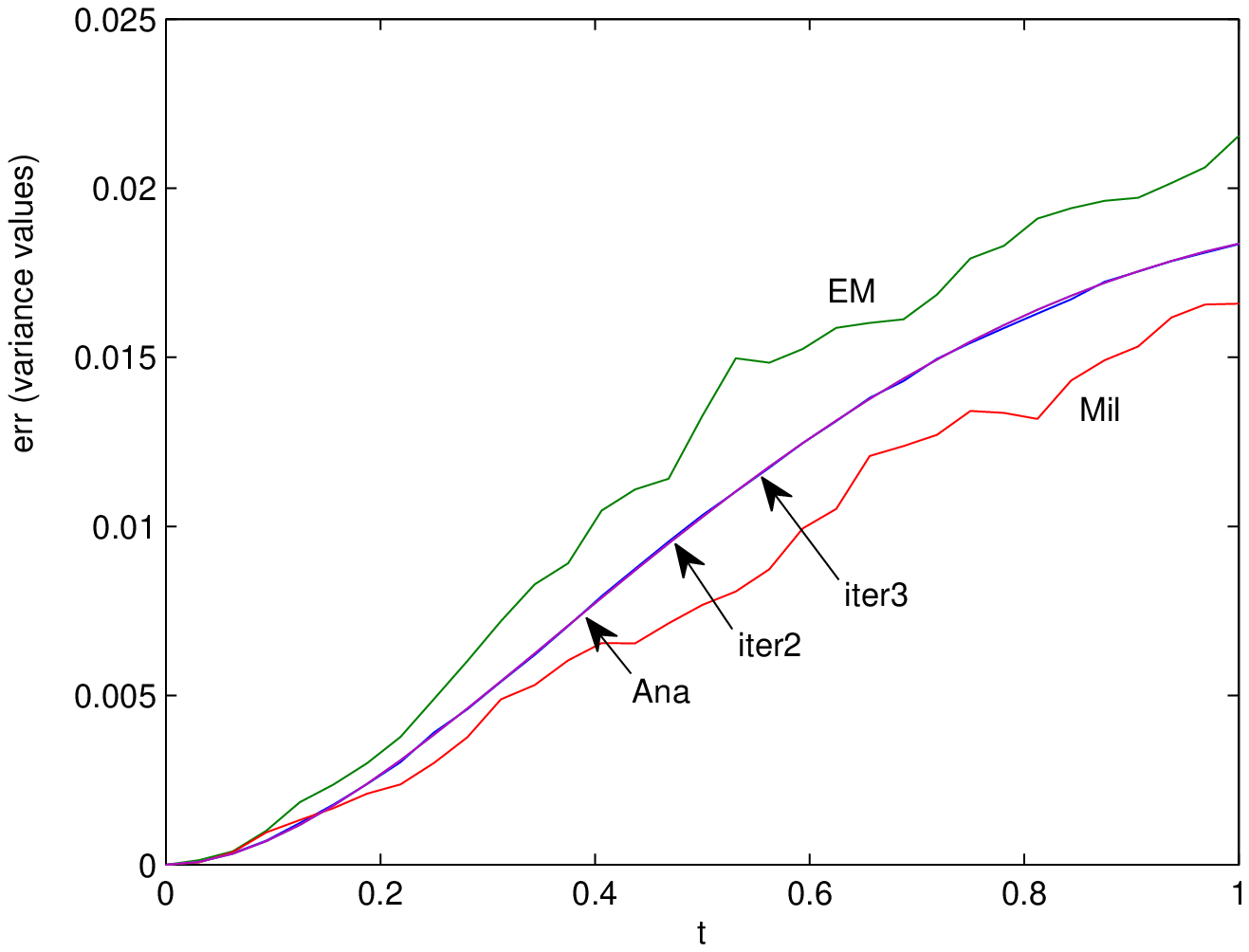} 
\end{center}
\caption{\label{lower_real} The upper right figure presents the results of the
EM, rS, sS-schemes (difference between exact and numerical solutions). 
The upper left figure present the results of the mean values (mean value of the difference between the exact and numerical solutions). The lower figure presents the variance of the schemes.}
\end{figure}

%

%

\begin{remark}
In the multiplicative noise example, we present the benefits of the 
iterative splitting schemes, which resolves the stochastic behavior more accurate as the standard schemes.
While Euler-Maruyama and Milstein schemes are explicit methods, the iterative approach is based on an implicit idea to relax the oscillations via additional iterative steps, see \cite{geiser_2011}. Based on these characteristics, we could reduce the numerical errors of the novel schemes with additional iterative steps.
\end{remark}

\subsection{Vectorial Multiplicative Noise (simple)}

In the following, we deal with a reduced $2 \time 2$ simple chemical reaction
model, but with non-commuting operators.

We deal first with an ordinary differential equation and separate 
the complex operator into two simpler operators: 
the $2 \times 2$ ordinary differential equation 
system given as:
\begin{eqnarray}
\label{num_7}
&& d {\bf y}(t) = A {\bf y}(t) + \sum_{j=1}^2 P_j {\bf y}(t) \; d W_j (t) , \\
&& A = \alpha_1 \left( \begin{array}{c c}
- \frac{1}{2} & 0 \\
  0 & - \frac{1}{2} 
\end{array} \right) , \\
&& P_1 = \alpha_2 \left( \begin{array}{c c}
\frac{3}{4} & \frac{1}{10}  \\
 0 &  - \frac{3}{4} 
\end{array} \right) , \\
&& P_2 = \alpha \left( \begin{array}{c c}
0 & \frac{9}{10}  \\
\frac{9}{10} & 0
\end{array} \right) , \\
 &&  dW(t) = (d W_1(t) , d W_2) \mbox{(stochastic vector)} , \nonumber \\
&&  y(0) = (1, \ldots, 1)^t \mbox{(initial conditions)} \nonumber ,
\end{eqnarray}
where $[P_1, P_2] \neq  0$. 
We have the time interval $t \in [0,T]$ and $m = 2$.

We apply a weak perturbation with $\alpha_2 = 0.01$ and a high perturbation with
$\alpha_2 =1.0$, for $\alpha_1= 1.0$ we apply a moderate convection.

We apply $T = 1$ and we have $N = 20$ time steps, means $\Delta t = T/N$.

For the testing the different numerical methods, we have the following
analytical solution, see \cite{platen2010}:
\begin{eqnarray}
\label{num_7}
&& {\bf y}(t^{n+1}) = \exp \left( A \Delta t - \frac{1}{2} \sum_{j=1}^m P_j P_j^t \Delta t + \sum_{j=1}^m P_j \Delta W_j \right) \;  {\bf y}(t^{n}),
\end{eqnarray}
where $\Delta W_j = (W_{t_{n+1}, j} - W_{t_n, j}) = \sqrt{\Delta t} N_j(0,1)$, where $N_j(0,1) = rand_j$, where we have $j = 1, \ldots, m$ normally distributed random variables. 

We apply the following numerical schemes:
\begin{itemize}
\item
The application of the standard Euler-Maruyama scheme is given as:
\begin{eqnarray}
&& y_{n+1} = y_n + A y_n \Delta t +  \sum_{j=1}^m P_j y_n \Delta W_j ,
\end{eqnarray}
for $n = 0, 1, \ldots, N-1$, $y_{0}= y_{t_0}$, $\Delta t = t_{n+1} - t_n$, $\Delta W = (W_{t_{n+1}, 1} - W_{t_n, 1}, \ldots, W_{t_{n+1}, m} - W_{t_n, m} ) = (\sqrt{\Delta t} N_1(0,1), \ldots, \sqrt{\Delta t} N_m(0,1)$, where $N_i(0,1) = rand_i$, where we have $i = 1, \ldots, m$ normally distributed random variables. 
\item
 Milstein scheme (without outer-diagonal entries)is given as:
\begin{eqnarray}
 y_{n+1} && = y_n +  A y_n \Delta t +  \sum_{j=1}^m P_j y_n \Delta W_j \nonumber  \\
     && + ( \sum_{i=1}^m \frac{1}{2} P_i P_i^t y_n \; ((\Delta W_i)^2 - \Delta t) ),
\end{eqnarray}
for $n = 0, 1, \ldots, N-1$, $y_{0}= y_{t_0}$.
\item
 Milstein scheme (with outer-diagonal entries) is given as:
\begin{eqnarray}
 y_{n+1} && = y_n +  A y_n \Delta t +  \sum_{j=1}^m P_j y_n \Delta W_j \nonumber  \\
     &&+ ( \sum_{i=1}^m \frac{1}{2} P_i P_i^t y_n \; ((\Delta W_i)^2 - \Delta t) ) \nonumber \\
    && + \sum_{i=1}^m  \sum_{j=i+1}^m  \frac{1}{2} [P_i , P_j] (J_{ji} - J_{ij}) y_n  , 
\end{eqnarray}
for $n = 0, 1, \ldots, N-1$, $y_{0}= y_{t_0}$. The commutator is given as
$[P_i, P_j] = P_i P_j - P_j P_i$.

Further the $J_{ij}$ are given as:
\begin{eqnarray}
 J_{ji} && = \frac{1}{2} J_j J_i - \frac{1}{2} (a_{i0} J_j - a_{j0} J_i) , 
\end{eqnarray}
with $J_i = \Delta W_i = (W_{t_{n+1}, i} - W_{t_n, i} = \sqrt{\Delta t} N_i(0,1)$
and the coefficients are given as:
$a_{i0} = \Delta \tilde{W}_i $, where $\Delta \tilde{W}_i = \sqrt{\frac{\Delta t}{2 \pi^2}} N_i(0,1)$, where $N_i(0,1) = rand_i$,

\item Iterative splitting scheme:

{\bf Version 1: 1 iterative steps}

We apply:

$X(0) \rightarrow X(\Delta t) \rightarrow  X(2 \Delta t) \ldots $

Zero iterative step:
\begin{eqnarray}
&& X_{0,n+1} = exp(A \Delta t + \sum_{j = 1}^m P_j \Delta W_j) X_{0,n} , 
\end{eqnarray}
where $X_{1,0} = y(0)$ and we have $N$ time-steps with $\Delta t = T/n$
and $t^{n+1} = t^n + \Delta t$ with $n = 0, \ldots, N-1$.

First iterative step:
\begin{eqnarray}
&& X_{1,n+1} =  exp(A \Delta t)  X_{1,n} \\
&&  + \int_{0}^{\Delta t} \exp(A (t -s))  \left( \sum_{i=1}^m  P_i X_{1, s}  d W_{i, s} \right) , \nonumber  \\
&& =  exp(A \Delta t)  X_{1,n} \\
&&  + \int_{0}^{\Delta t} \exp(A (t -s)) \left( \sum_{i=1}^m  P_i \exp(A s + \sum_{j = 1}^m P_j W_{j, s}) d W_{i, s}   \right) X_{1,n}  , \nonumber   \\
&& =  exp(A \Delta t)  X_{1,n} \\
&& +  \sum_{j=1}^m P_j X_{1,n} \Delta W_j \nonumber  \\
     && + ( \sum_{i=1}^m \frac{1}{2} P_i P_i^t X_{1,n} \; ((\Delta W_i)^2 - \Delta t) ), \nonumber \\
    && + \sum_{i=1}^m  \sum_{j=i+1}^m  \frac{1}{2} [P_i , P_j] (J_{ji} - J_{ij})  X_{1,n}  , \nonumber
\end{eqnarray}
for $n = 0, 1, \ldots, N-1$, $y_{0}= y_{t_0}$. The commutator is given as
$[P_i, P_j] = P_i P_j - P_j P_i$.

Further the $J_{ij}$ are given as:
\begin{eqnarray}
 J_{ji} && = \frac{1}{2} J_j J_i - \frac{1}{2} (a_{i0} J_j - a_{j0} J_i) , 
\end{eqnarray}
with $J_i = \Delta W_i = (W_{t_{n+1}, i} - W_{t_n, i} = \sqrt{\Delta t} N_i(0,1)$
and the coefficients are given as:
$a_{i0} = \Delta \tilde{W}_i $, where $\Delta \tilde{W}_i = \sqrt{\frac{\Delta t}{2 \pi^2}} N_i(0,1)$, where $N_i(0,1) = rand_i$,

We obtain the Milstein scheme with outer-diagonal entries.

{\bf Version 2: 2 iterative steps}

Second iterative step:
\begin{eqnarray}
&& X_{2,n+1} =  exp(A \Delta t)  X_{2,n} \\
&&  + \int_{0}^{\Delta t} \exp(A (t -s))  \left( \sum_{i=1}^m  P_i X_{1, s}  d W_{i, s} \right) , \nonumber  \\
&& =  exp(A \Delta t)  X_{2,n} \\
&&  + \int_{0}^{\Delta t} \exp(A (t -s)) \left( \sum_{i=1}^m  P_i \bigg(  \exp(A s)  X_{1,n} \bigg. \right. \nonumber \\
&& \left. \bigg. + \int_{0}^{s} \exp(A (s -s_1))  ( \sum_{j=1}^m  P_j X_{1, s_1}  d W_{j, s_1} )   \bigg) d W_{i, s} \right)  , \nonumber   \\
&& =  exp(A \Delta t)  X_{2,n} \\
&& +  \sum_{j=1}^m P_j X_{2,n} \Delta W_j \nonumber  \\
     && + ( \sum_{i=1}^m \frac{1}{2} P_i P_i^t X_{2,n} \; ((\Delta W_i)^2 - \Delta t) ), \nonumber \\
    && + \sum_{i=1}^m  \sum_{j=i+1}^m  \frac{1}{2} [P_i , P_j] (J_{ji} - J_{ij})  X_{2,n}  , \nonumber \\
    && + \sum_{i=1}^m   \frac{1}{2} (P_i ( P_i^t P_i)^t) \left( (\frac{1}{3} (\Delta W_i)^2 - \Delta t) \Delta W_i \right) \; X_{2,n}  , \nonumber
\end{eqnarray}
for $n = 0, 1, \ldots, N-1$, $y_{0}= y_{t_0}$. The commutator is given as
$[P_i, P_j] = P_i P_j - P_j P_i$.

Further the $J_{ij}$ are given as:
\begin{eqnarray}
 J_{ji} && = \frac{1}{2} J_j J_i - \frac{1}{2} (a_{i0} J_j - a_{j0} J_i) , 
\end{eqnarray}
with $J_i = \Delta W_i = (W_{t_{n+1}, i} - W_{t_n, i} = \sqrt{\Delta t} N_i(0,1)$
and the coefficients are given as:
$a_{i0} = \Delta \tilde{W}_i $, where $\Delta \tilde{W}_i = \sqrt{\frac{\Delta t}{2 \pi^2}} N_i(0,1)$, where $N_i(0,1) = rand_i$,

We obtain a version which is nearly $\OT(\Delta t^{1.5})$ (and more accurate as the 
Milstein scheme).

\end{itemize}

For different case of strong and weak perturbations, we have the following Figures \ref{perturb_1}.
\begin{figure}[ht]
\begin{center}  
\includegraphics[width=5.0cm,angle=-0]{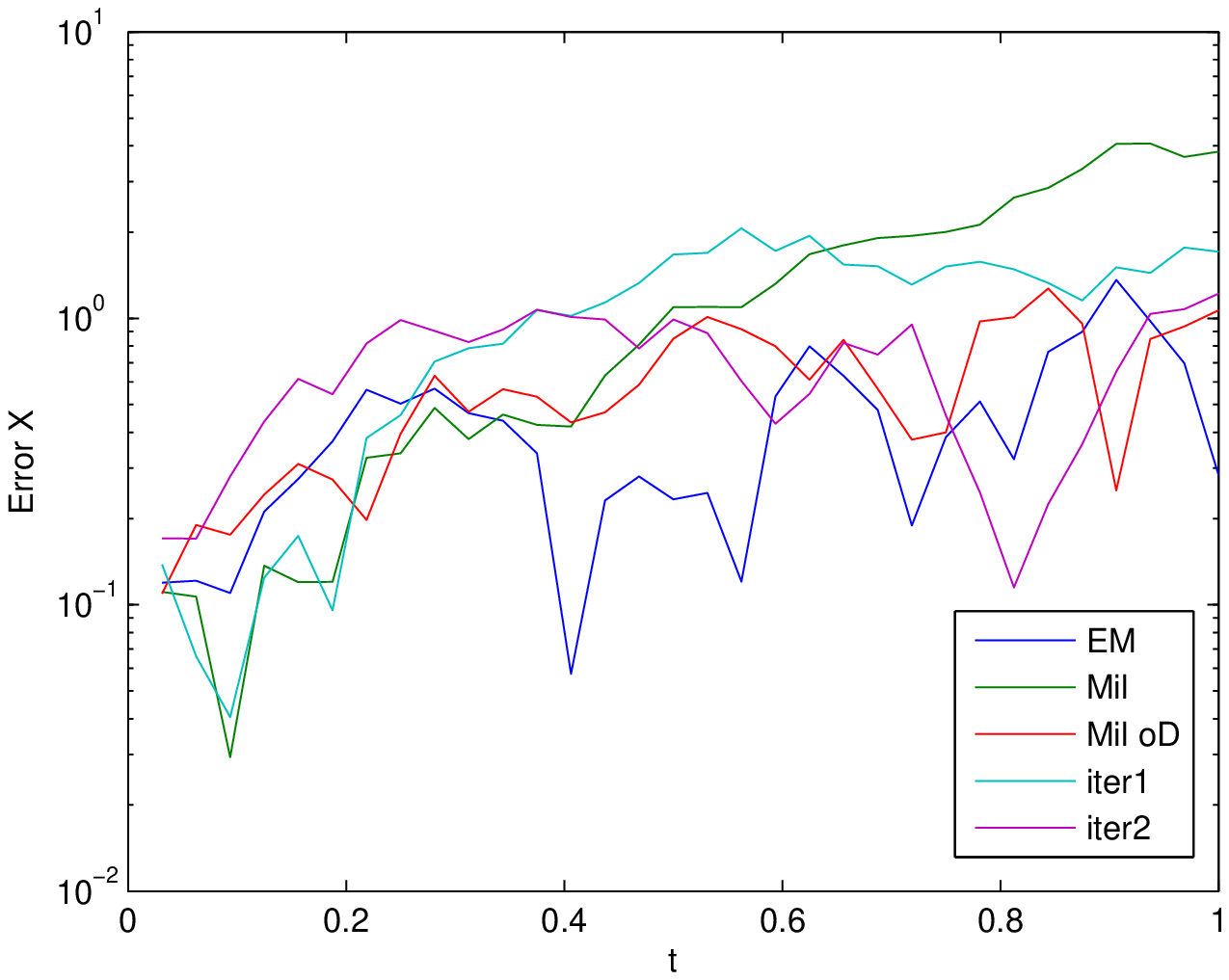} 
\includegraphics[width=5.0cm,angle=-0]{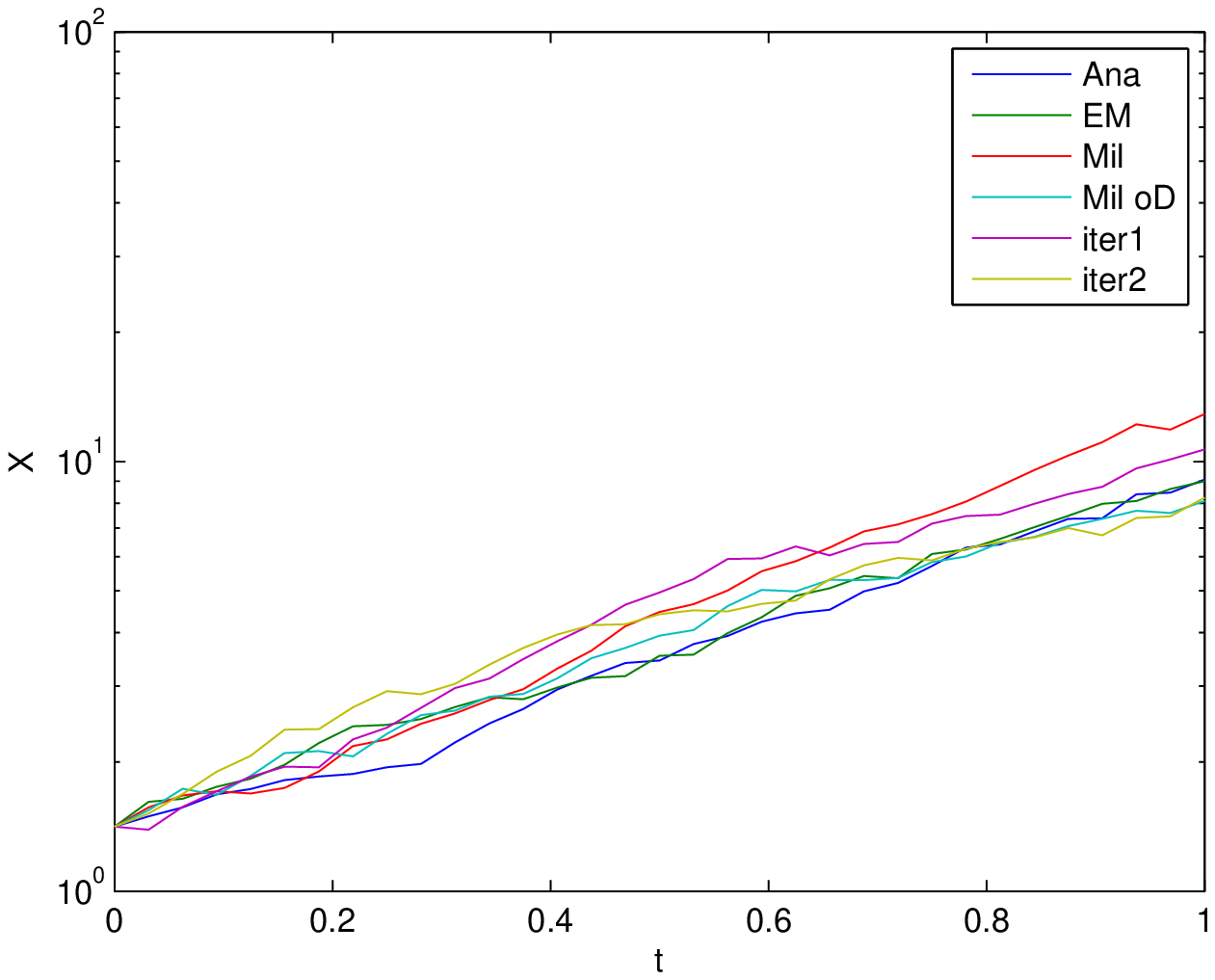} \\
\includegraphics[width=5.0cm,angle=-0]{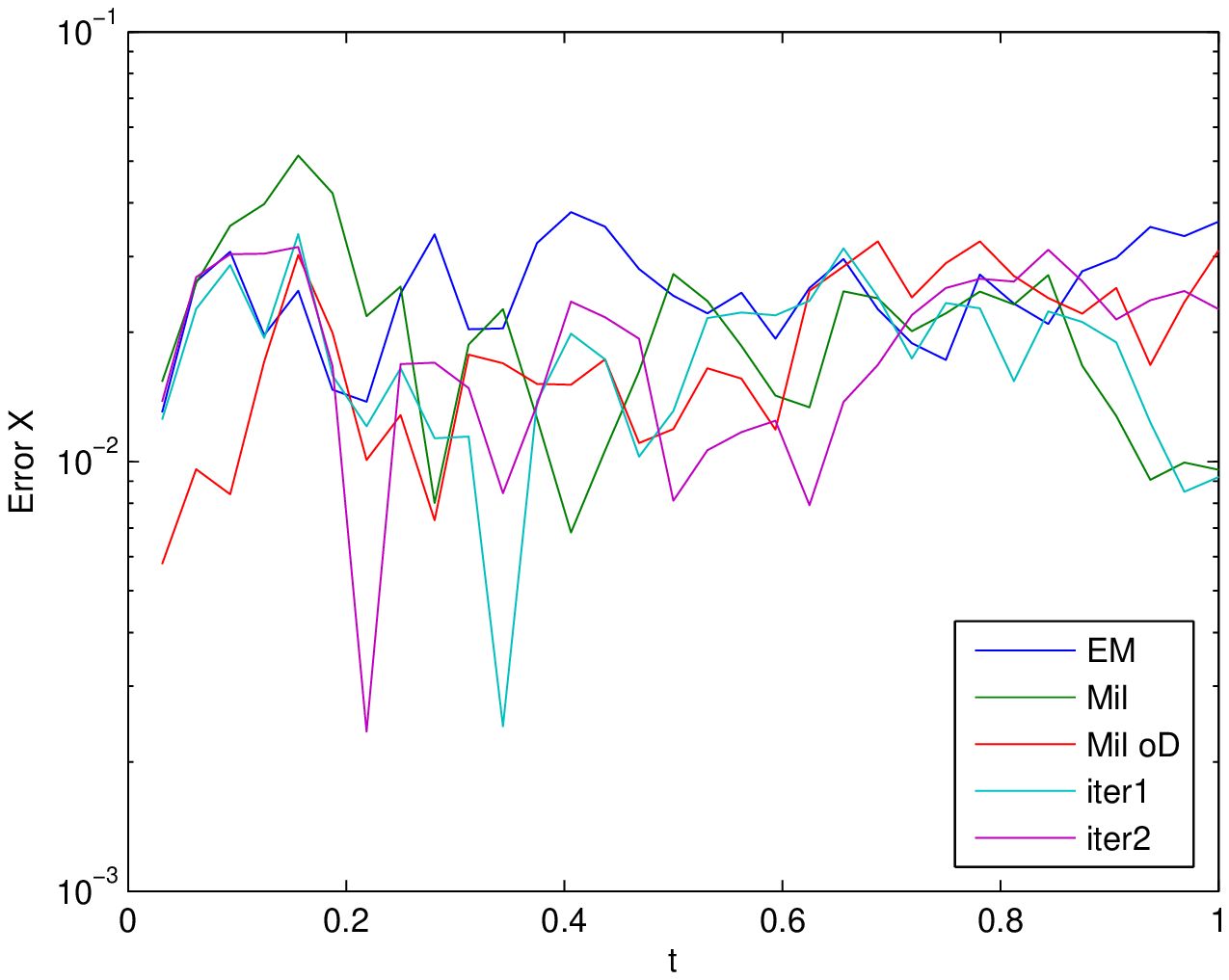} 
\includegraphics[width=5.0cm,angle=-0]{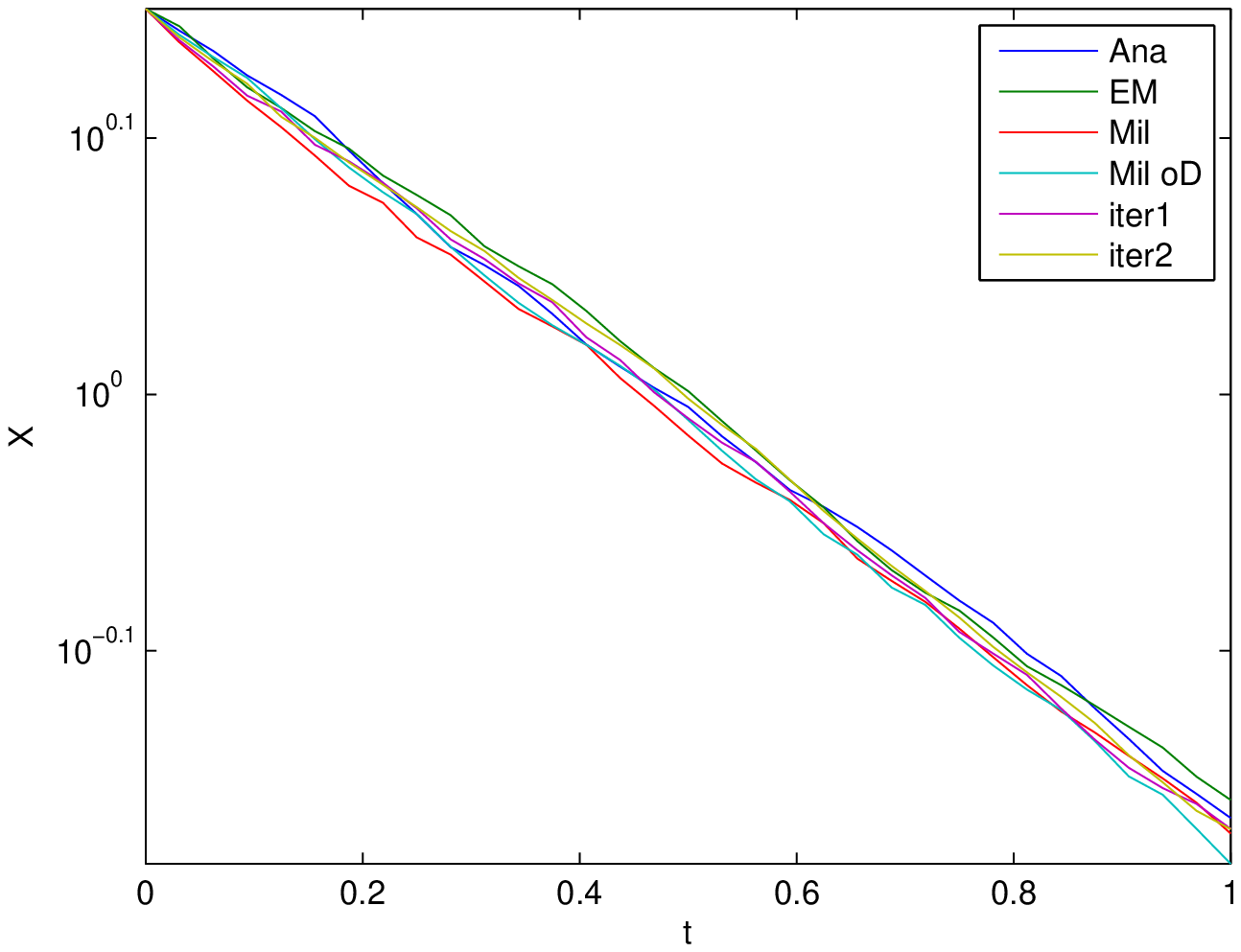}
\end{center}
\caption{\label{perturb_1}  The upper figures presents the results of the strong perturbations $\alpha_1 = 1.0, \; \alpha_2 = 1.0$ (left figure the numerical error, right figure the numerical solution). 
The right figures presents the results for the weak perturbations $\alpha_1 = 1.0, \; \alpha_2 = 0.1$ (left figure the numerical error, right figure the numerical solution).
We apply the following schemes: EM: Euler-Maruyama, Milstein, Iter1: 1-iterative steps, Iter2: 2-iterative steps.
}
\end{figure}

\begin{remark}
We can also verify the benefit of the additional terms, which are necessary to 
resolve the vectorial stochastics. Here, we obtain for both non-iterative and
iterative schemes higher order results for the extension of the schemes. In the numerical
implementations, we also receive the benefit of the exponential matrices related to the
iterative schemes, see \cite{geiser_2011_2}.
\end{remark}

\subsection{Vectorial Multiplicative Noise (non-commutative): Real-life example}

In the following, we deal with a simple chemical reaction
model, but in a vectorial manner.

We deal first with an ordinary differential equation and separate 
the complex operator into two simpler operators: 
the $m \times m$ ordinary differential equation 
system given as:
\begin{eqnarray}
\label{num_7}
&& d {\bf y}(t) = A {\bf y}(t) + \sum_{j=1}^m P_j {\bf y}(t) \; d W_j (t) , \\
&& A = \left( \begin{array}{c c c c}
 - \lambda_{1,1} & \lambda_{2,1}  & \dots & \lambda_{1,m}   \\
    \lambda_{2,1} & -  \lambda_{2,2} &  \dots  & \lambda_{2,m}  \\
\vdots \\
   \lambda_{m,1} &  \lambda_{m,2} & \dots & - \lambda_{m,m} \\
\end{array} \right) =  \left( \begin{array}{c c c c}
 - 1 & 0  & \dots & 0   \\
   \frac{1}{m} & - 1 &  \dots  & 0  \\
\vdots \\
   \frac{1}{m} & \frac{1}{m} & \dots & - 1 \\
\end{array} \right) \nonumber \\
&& P_1 = \left( \begin{array}{c c c c}
 \sigma_{1,1} &  \sigma_{1,2} & \dots &  \sigma_{1,m}    \\
  \sigma_{2,1} &  \sigma_{2,2} &  \dots  &  \sigma_{2,m}  \\
\vdots \\
  \sigma_{m,1}  &  \sigma_{m,2} & \dots &  \sigma_{m,m} \\
\end{array} \right)  = 0.05 \; \left( \begin{array}{c c c c}
 1 &  0  & \dots & 0    \\
 \frac{1}{m} & 1 &  \dots  & 0   \\
\vdots \\
 \frac{1}{m}  & \frac{1}{m} & \dots & 1  \\
\end{array} \right) \nonumber \\
&& P_2 = 0.05 \; \left( \begin{array}{c c c c}
 \tau_{1,1} &  \tau_{1,2} & \dots &  \tau_{1,m}    \\
  \tau_{2,1} &  \tau_{2,2} &  \dots  &  \tau_{2,m}  \\
\vdots \\
  \tau_{m,1}  &  \tau_{m,2} & \dots &  \tau_{m,m} \\
\end{array} \right)  = \left( \begin{array}{c c c c}
 1 &  \frac{1}{m} & \dots &  \frac{1}{m}    \\
 0  & 1 &  \dots  &  \frac{1}{m}  \\
\vdots \\
  0  & 0 & \dots & 1  \\
\end{array} \right) \nonumber \\
 &&  dW(t) = (d W_1(t), d W_2) \mbox{(stochastic vector)} , \nonumber \\
&&  y(0) = (1, \ldots, 1)^t \mbox{(initial conditions)} \nonumber , \\
\end{eqnarray}
where $\lambda_{11} \ldots \lambda_{m, m} \in \R^+$ are the decay
factors and $\sigma_{11}, \ldots,  \sigma_{m,m}, \tau_{11}, \ldots, \tau_{m,m} \in \R^+$ are the 
parameters of the perturbations. 
We have the time interval $t \in [0,T]$ and $m \in \N$.

For the testing the different numerical methods, we have the following
analytical solution, see \cite{platen2010}:
\begin{eqnarray}
\label{num_7}
&& {\bf y}(t^{n+1}) = \exp \left( A \Delta t - \frac{1}{2} \sum_{j=1}^m P_j P_j^t \Delta t + \sum_{j=1}^m P_j \Delta W_j \right) \;  {\bf y}(t^{n}),
\end{eqnarray}
where $\Delta W_j = (W_{t_{n+1}, j} - W_{t_n, j}) = \sqrt{\Delta t} N_j(0,1)$, where $N_j(0,1) = rand_j$, where we have $j = 1, \ldots, m$ normally distributed random variables. 

We apply the following numerical schemes:
\begin{itemize}
\item
The application of the standard Euler-Maruyama scheme is given as:
\begin{eqnarray}
&& y_{n+1} = y_n + A y_n \Delta t +  \sum_{j=1}^m P_j y_n \Delta W_j ,
\end{eqnarray}
for $n = 0, 1, \ldots, N-1$, $y_{0}= y_{t_0}$, $\Delta t = t_{n+1} - t_n$, $\Delta W = (W_{t_{n+1}, 1} - W_{t_n, 1}, \ldots, W_{t_{n+1}, m} - W_{t_n, m} ) = (\sqrt{\Delta t} N_1(0,1), \ldots, \sqrt{\Delta t} N_m(0,1)$, where $N_i(0,1) = rand_i$, where we have $i = 1, \ldots, m$ normally distributed random variables. 
\item
 Milstein scheme (without outer-diagonal entries)is given as:
\begin{eqnarray}
 y_{n+1} && = y_n +  A y_n \Delta t +  \sum_{j=1}^m P_j y_n \Delta W_j \nonumber  \\
     && + ( \sum_{i=1}^m \frac{1}{2} P_i P_i^t y_n \; ((\Delta W_i)^2 - \Delta t) ),
\end{eqnarray}
for $n = 0, 1, \ldots, N-1$, $y_{0}= y_{t_0}$.
\item
 Milstein scheme (with outer-diagonal entries) is given as:
\begin{eqnarray}
 y_{n+1} && = y_n +  A y_n \Delta t +  \sum_{j=1}^m P_j y_n \Delta W_j \nonumber  \\
     &&+ ( \sum_{i=1}^m \frac{1}{2} P_i P_i^t y_n \; ((\Delta W_i)^2 - \Delta t) ) \nonumber \\
    && + \frac{1}{2} \sum_{i=1}^m  \sum_{j=i+1}^m  \frac{1}{2} [P_i , P_j] (J_{ji} - J_{ij}) y_n  , 
\end{eqnarray}
for $n = 0, 1, \ldots, N-1$, $y_{0}= y_{t_0}$. The commutator is given as
$[P_i, P_j] = P_i P_j - P_j P_i$.

Further the $J_{ij}$ are given as:
\begin{eqnarray}
 J_{ji} && = \frac{1}{2} J_j J_i - \frac{1}{2} (a_{i0} J_j - a_{j0} J_i) , 
\end{eqnarray}
with $J_i = \Delta W_i = (W_{t_{n+1}, i} - W_{t_n, i}) = \sqrt{\Delta t} N_i(0,1)$
and the coefficients are given as:
$a_{i0} = \Delta \tilde{W}_i $, where $\Delta \tilde{W}_i = \frac{\sqrt{\Delta t}}{2 \pi^2} N_i(0,1)$, where $N_i(0,1) = rand_i$,

\item Iterative splitting scheme is given as in the previous example.

\end{itemize}

In the following, we have the computations of the non-commutative example.

The solution of the 10 species with the iterative scheme (2 steps) 
and the different schemes for the 10-th species is given in \ref{solu_1}.
\begin{figure}[ht]
\begin{center}  
\includegraphics[width=4.0cm,angle=-0]{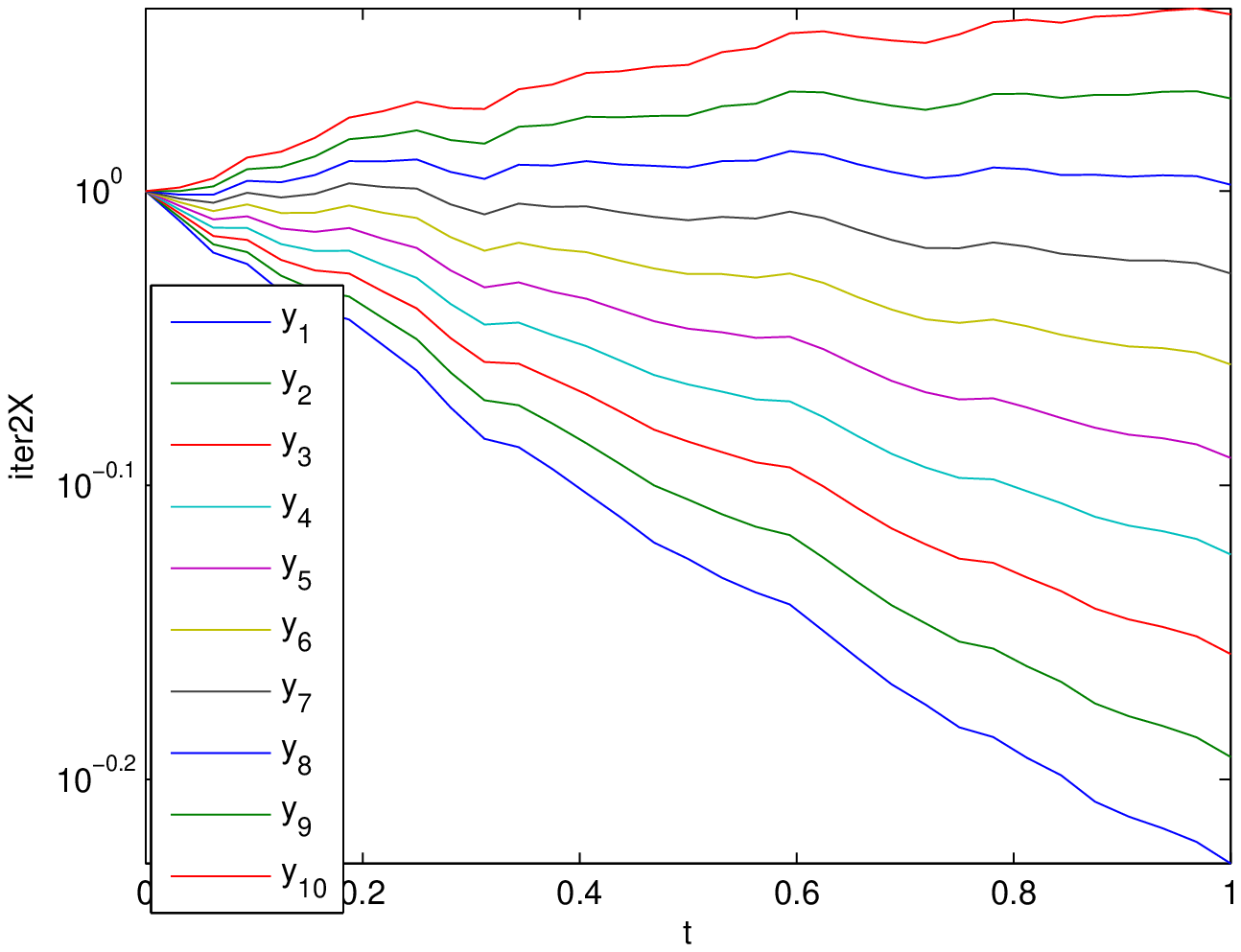}
\includegraphics[width=4.0cm,angle=-0]{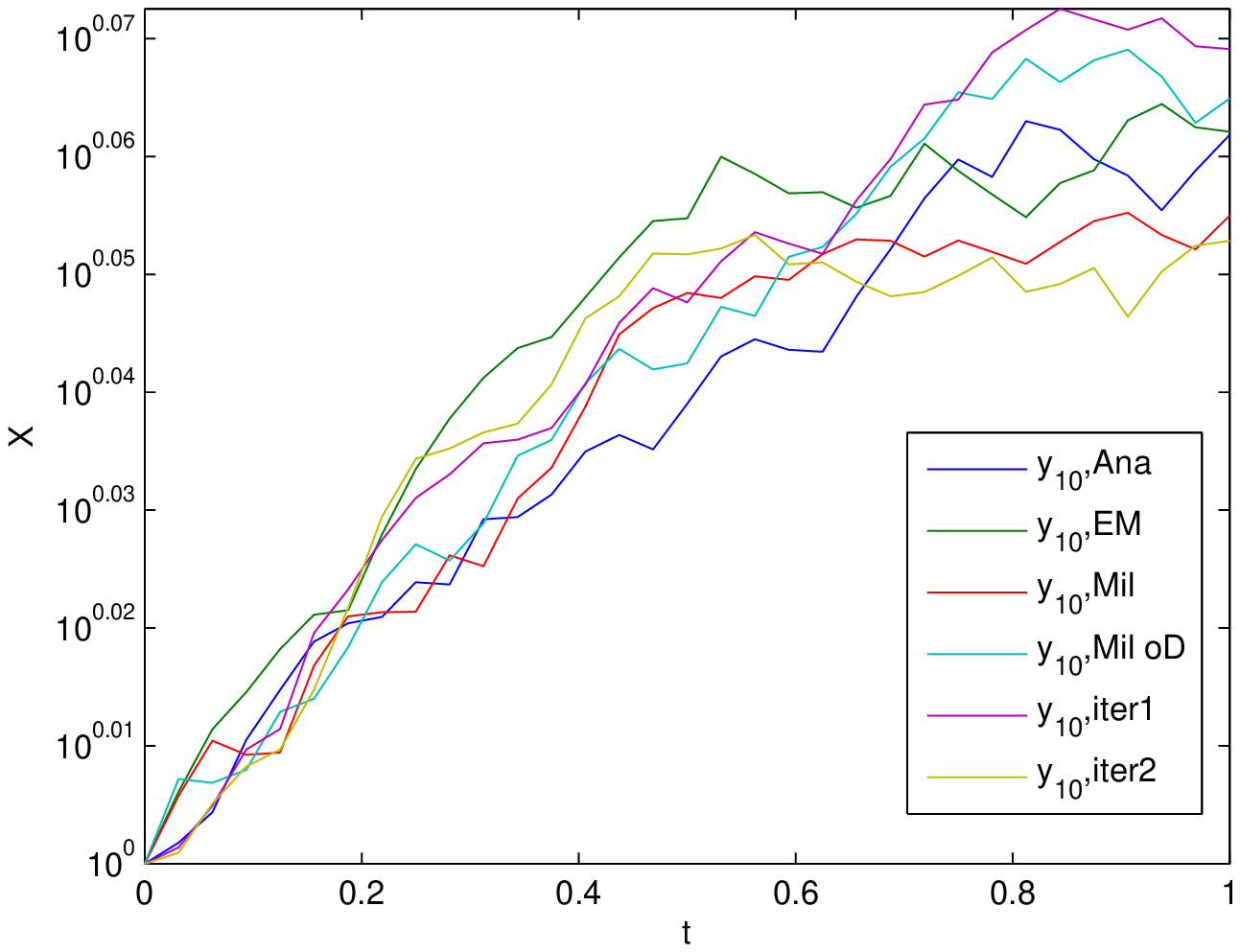}
\end{center}
\caption{\label{solu_1} The solution of the 10 species (left hand side) and the
solutions for the 10-th species with the different schemes (right hand side).}
\end{figure}

The errors of the different scheme with respect to the
$L_2$-norm, weak and strong error is given in Figure \ref{solu_2}.
\begin{figure}[ht]
\begin{center}  
\includegraphics[width=4.0cm,angle=-0]{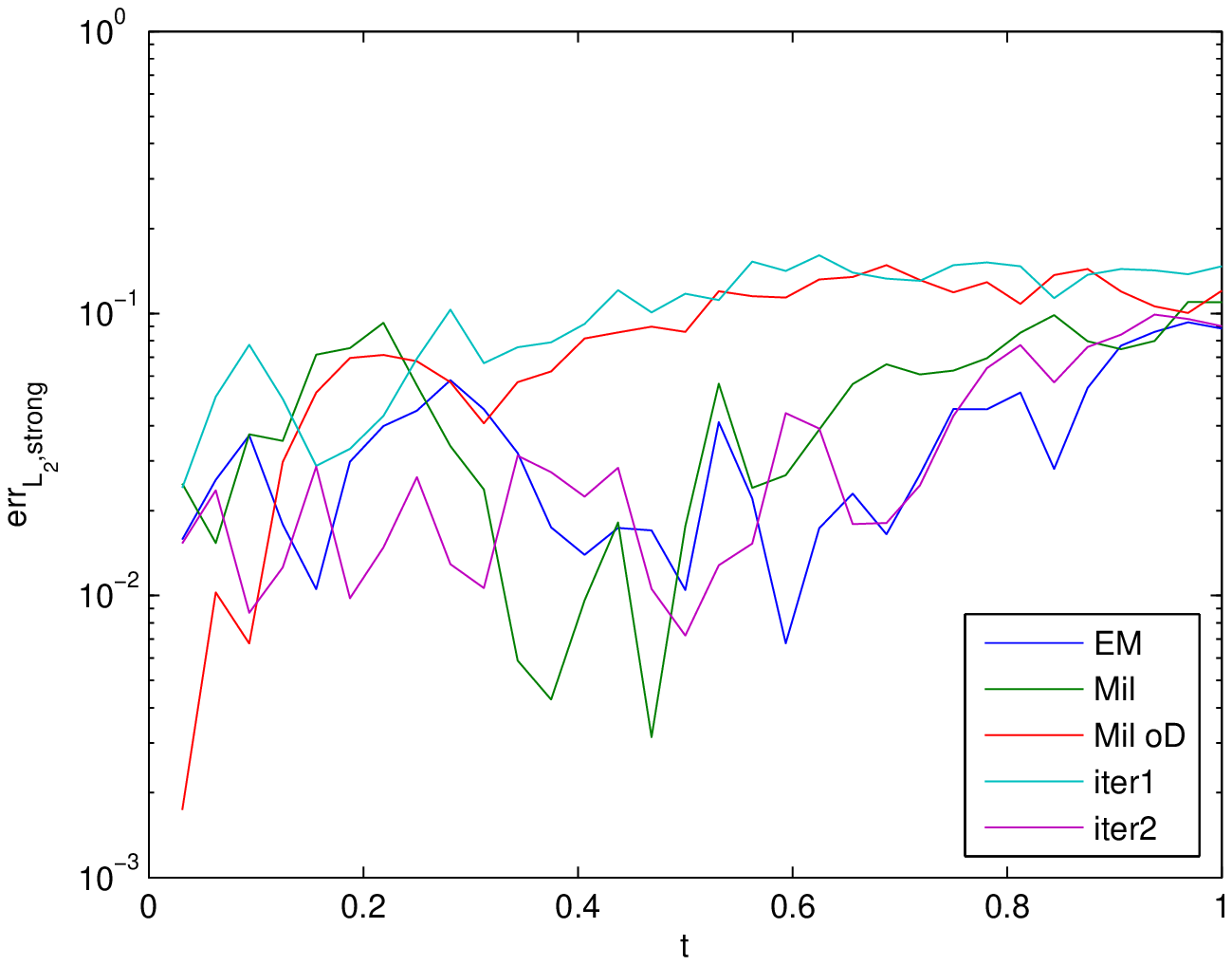}
\includegraphics[width=4.0cm,angle=-0]{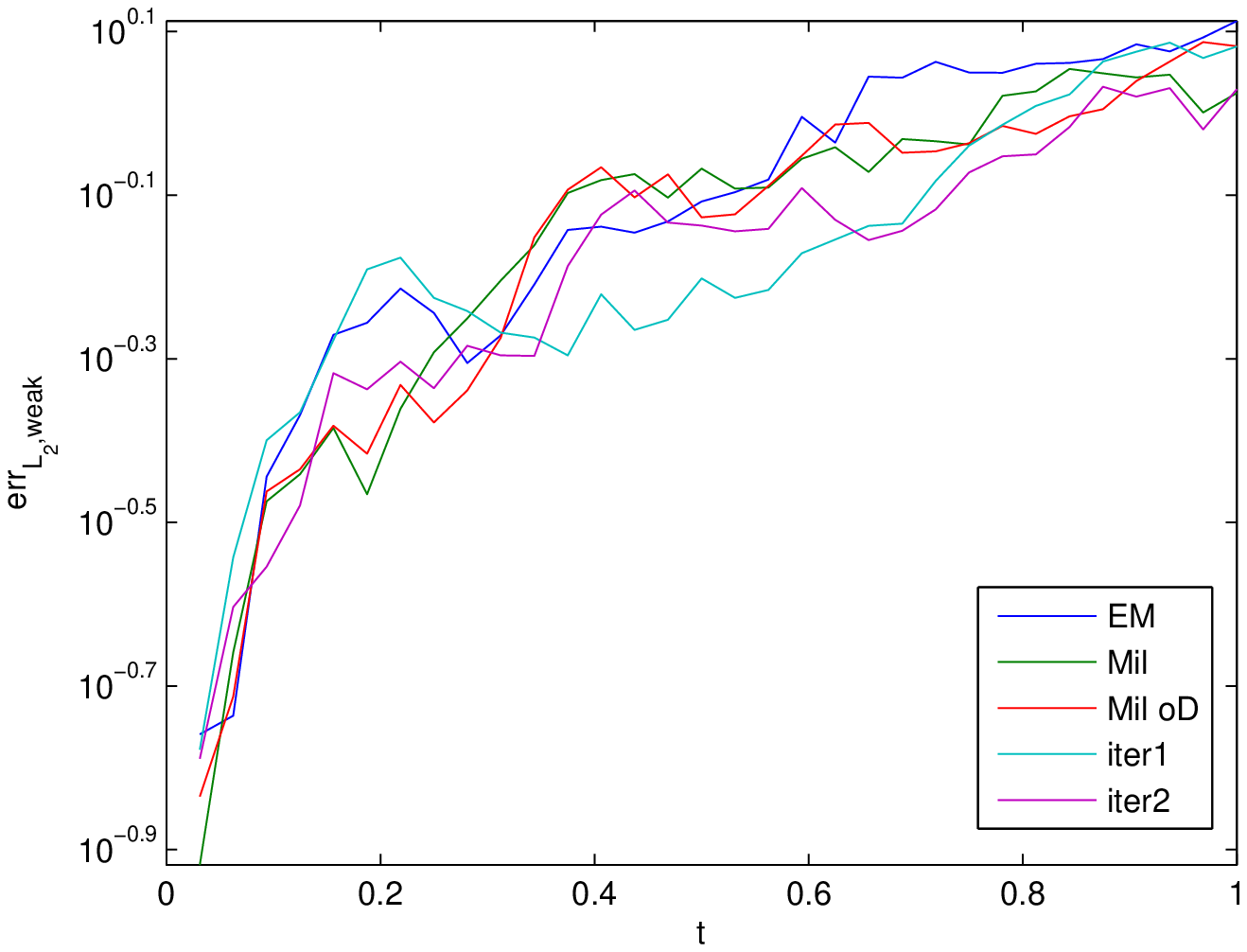}
\end{center}
\caption{\label{solu_2} The strong error of the numerical solution is given in the left figure.
The weak error of the numerical solution is given in the right figure.}
\end{figure}

\begin{remark}
For larger matrices, we also obtain a benefit, when we apply iterative methods.
We are more flexible as for the standard schemes, while we can increase the
order of the method with additional iterative steps. Further, the computational amount
for additional steps are marginal based on the recursive effect of the
iterative splitting scheme. Therefore, we can resolve the solution in the
same accuracy as a Milstein scheme with additional multiple integral terms, see \cite{kloeden1992}.
\end{remark}

\subsection{Coulomb test-particle problem (vectorial problem of the linearized Langevin equations)}

In the next example, we deal with a real-life problem, which models
the characteristics of a collision process, see \cite{dimits2013}.

We apply the following nonlinear SDE problem:
\begin{eqnarray}
&& dv(t) = F_d(v) dt + \sqrt{2D_v(v)} dW_v(t) , \\
&& d\mu(t) = - 2 D_a(v) \mu dt + \sqrt{2D_a(v) (1 - \mu^2)} dW_{\mu}(t) , \\
&& d\phi(t) = \sqrt{\frac{2D_a(v)}{(1 - \mu^2)}} dW_{\phi}(t) ,
\end{eqnarray}
where the functions and the derivatives of the convection and 
diffusion operators are given as:
\begin{eqnarray}
& D_v(v) = \frac{1}{2}  \frac{1}{v + 1} , & \frac{\partial D_v}{\partial v} = - \frac{1}{2}  (v + 1)^{-2} , \\
& F_d(v) = - \frac{1}{2}  \frac{1}{v + 1} , &  \frac{\partial F_d}{\partial v} = \frac{1}{2}  (v + 1)^{-2} ,  \\
& D_a(v) = \frac{1}{2}  \frac{1}{v + 1}  , & \frac{\partial D_a}{\partial v} = - \frac{1}{2}  (v + 1)^{-2} ,
\end{eqnarray}
and where we assume that the initial conditions are given as $v_0 = 1.0$, $\mu_0 = 1.0$ and $\phi_0 = 1.0$.

The notation of the equation in vectorial form is given as:
\begin{eqnarray}
&& d {\bf v}(t) = {\bf a}({\bf v}) dt + B({\bf v}) d{\bf W}_{\bf v}(t) , 
\end{eqnarray}
where $ {\bf v}(t) = (v, \mu, \phi)^t$ and the vectors and matrix are given as:

\begin{eqnarray}
&& {\bf a}({\bf v}) = \left( \begin{array}{c}
   F_d(v) \\
 - 2 D_a(v) \mu \\  
      0 
\end{array} \right),
d {\bf W}_{\bf v} = \left( \begin{array}{c}
 d W_v  \\
  d W_{\mu} \\ 
   d W_{\phi}
\end{array} \right), \\
&& B({\bf v}) = \left( \begin{array}{c c c }
 \sqrt{2D_v(v)}  & 0 & 0  \\
0 & \sqrt{2D_a(v) (1 - \mu^2)} & 0  \\
0 & 0 &  \sqrt{\frac{2D_a(v)}{(1 - \mu^2)}} 
\end{array} \right),
\end{eqnarray}

We apply the following numerical schemes:
\begin{itemize}
\item
The application of the standard Euler-Maruyama scheme is given as:
\begin{eqnarray}
&& v_{n+1} = v_n + F(v_n) \Delta t +  \sqrt{2 D(v_n)} \Delta W_v , \\
&& \mu_{n+1} = \mu_{n} - 2 D_a(v_n) \mu_n \Delta t + \sqrt{2  D_a(v_n) (1 - \mu_n^2)} \Delta W_{\mu} , \\
&& \phi_{n+1} = \phi_{n} + \sqrt{\frac{2  D_a(v_n)}{(1 - \mu_n^2)}} \Delta W_{\phi} , 
\end{eqnarray}
for $n = 0, 1, \ldots, N-1$, $v_{0}= v(0), \mu_0 = \mu(0), \phi_0 = \phi(0)$, 
$\Delta t = t_{n+1} - t_n$, $\Delta W_{i} = W_{i, t_{n+1}} - W_{i, t_n} = \sqrt{\Delta t} N_i(0,1)$, where $N_i(0,1) = rand$, $i = \{v, \mu, \phi\}$ are three independent normally distributed random variables. 
\item
 Milstein scheme is given as:
\begin{eqnarray}
 v_{n+1} && = v_n +  F(v_n) \Delta t +  \sqrt{2 D(v_n)} (\Delta W) \nonumber  \\
     && + \frac{\partial D(v)}{\partial v}|_{v_n}  \frac{1}{2} ((\Delta W)^2 - \Delta t), \\
 \mu_{n+1} && = \mu_{n} - 2 D_a(v_n) \mu_n \Delta t + \sqrt{2  D_a(v_n) (1 - \mu_n^2)} \Delta W_{\mu} \nonumber \\
 && - 2 \mu_n  D_a(v_n) \frac{1}{2}(\Delta W^2_{\mu} - \Delta t) \nonumber \\
 && +  \sqrt{\frac{D(v_n)}{D_a(v_n)}} \sqrt{(1 - \mu^2_n)}   \frac{\partial D_a(v)}{\partial v}|_{v_n}  A_{v, \mu} , \\
 \phi_{n+1} && =  \phi_{n} + \sqrt{\frac{2  D_a(v_n)}{(1 - \mu_n^2)}} \Delta W_{\phi} \nonumber  \\
&&  +  \sqrt{\frac{D(v_n)}{D_a(v_n)}} \frac{1}{\sqrt{(1 - \mu^2_n)}}   \frac{\partial D_a(v)}{\partial v}|_{v_n}  A_{v, \phi} \nonumber \\
&& +  \frac{2 D_a(v_n) \mu_n}{(1 - \mu^2_n)} A_{\mu, \phi} 
\end{eqnarray}
for $n = 0, 1, \ldots, N-1$, $v_{0}= v(0), \mu_0 = \mu(0), \phi_0 = \phi(0)$, 
$\Delta t = t_{n+1} - t_n$, $\Delta W_{i} = W_{i, t_{n+1}} - W_{i, t_n} = \sqrt{\Delta t} N_i(0,1)$, where $N_i(0,1) = rand$, $i = \{v, \mu, \phi\}$ are three independent normally distributed random variable.

The iterated Ito integral, which is related to Levy areas \cite{scheicher2007}, and given as:
\begin{eqnarray}
A_{k, l} = \int_{t^{n}}^{t^{n+1}}  d W_l(s) \int_{t^n}^{s} d W_k(\xi),
\end{eqnarray}
where we have for the outer-Diagonal case $k\neq l$:
\begin{eqnarray}
 A_{k,l} && = \frac{1}{2} J_k J_l - \frac{1}{2} (a_{l0} J_k - a_{k0} J_l) , 
\end{eqnarray}
with for $i = k,l$, we have  $J_i = \Delta W_i = (W_{t_{n+1}, i} - W_{t_n, i} = \sqrt{\Delta t} N_i(0,1)$
and the coefficients are given as:
$a_{i0} = \Delta \tilde{W}_i $, where $\Delta \tilde{W}_i = \sqrt{\frac{\Delta t}{2 \pi^2}} N_i(0,1)$, where $N_i(0,1) = rand_i$, see \cite{kloeden1992}.

\item Iterative splitting scheme:

We apply the following linearization techniques of the convective part and iterate via the diffusive part.

\begin{enumerate}
\item Fixpoint iterative version with simple relaxation of the nonlinear part is applied as:
\begin{eqnarray}
\label{iter_1_1}
&& d{\bf v}_{i+1}(t) = \hat{A}({\bf v}_{i})  {\bf v}_{i+1} dt + B({\bf v}_i)  d{\bf W}(t) ,
\end{eqnarray}
with the solution vector ${\bf v}_{i}(t) = (v_i(t), \mu_i(t), \phi_i(t))^t$.

Furthermore, the linearized matrix is given as
\begin{eqnarray}
\label{num_7}
&& \hat{A}({\bf v}_{i}) = 
\begin{bmatrix} \frac{F_v(v_i)}{v_{i}} & 0 & 0 \\
\\ 
0 & - 2 D_a(v_i) & 0 \\ 
\\
0 & 0 & 0 \end{bmatrix} ,
\end{eqnarray}
Then the fixpoint scheme is given as:
\begin{eqnarray}
\label{iter_1_1}
&& {\bf v}_{i+1}(t^{n+1}) = \exp(\hat{A}({\bf v}_{i}(t^{n+1})) \Delta t) \; {\bf v}(t^n) \nonumber \\
&& + \int_{t^n}^{t^{n+1}}  \exp(\hat{A}({\bf v}_{i}(t^{n+1})) \;( t^{n+1} - s ) \;  B({\bf v}_i(s))  d{\bf W}_{\bf v}(s) . 
\end{eqnarray}
where the integral is computed as: \\
1.) Trapezoidal-rule:
\begin{eqnarray}
\label{iter_1_1}
&& \int_{t^n}^{t^{n+1}}  \exp(\hat{A}({\bf v}_{i}(t^{n+1})) \;( t^{n+1} - s )) \;  B({\bf v}_i(s))  d{\bf W}_{\bf v}(s) \\
&& = \frac{1}{2} ({\bf W}_{\bf v}(t^{n+1}) - {\bf W}_{\bf v}(t^n))  \bigg(  B({\bf v}_i(t^{n+1}))
\\
&& +  \exp(\hat{A}({\bf v}_{i}(t^{n+1})) \; \Delta t ) \;  B({\bf v}_i(t^n))  \bigg) , \nonumber
\end{eqnarray}
$\Delta t= t^{n+1} - t^n$ and \\
 $({\bf W}(t^{n+1}) - {\bf W}(t^{n})) = ( rand_1 \sqrt{\Delta t},  rand_2 \sqrt{\Delta t},  rand_3 \sqrt{\Delta t} )^t$ , \\
2.) Simpson-rule 
\begin{eqnarray}
\label{iter_1_1}
&& \int_{t^n}^{t^{n+1}}  \exp(\hat{A}({\bf v}_{i}(t^{n+1})) \;( t^{n+1} - s )) \;  B({\bf v}_i(s))  d{\bf W}_{\bf v}(s) \\
&& = \frac{1}{6} ({\bf W}_{\bf v}(t^{n+1}) - {\bf W}_{\bf v}(t^n))  \bigg(  B({\bf v}_i(t^{n+1}))
\\
&& + 4 \exp(\hat{A}({\bf v}_{i}(t^{n} + \Delta t/2 )) \; \Delta t/2 ) \;  B({\bf v}_i(t^n + \Delta t/2)) \nonumber \\
&& + +  \exp(\hat{A}({\bf v}_{i}(t^{n+1})) \; \Delta t ) \;  B({\bf v}_i(t^n))  \bigg) , \nonumber
\end{eqnarray}
$\Delta t= t^{n+1} - t^n$ and \\
 $({\bf W}(t^{n+1}) - {\bf W}(t^{n})) = ( rand_1 \sqrt{\Delta t},  rand_2 \sqrt{\Delta t},  rand_3 \sqrt{\Delta t} )^t$ , \\

\item Fixpoint iterative version with Taylor expansion of the nonlinear part is applied as:
\begin{eqnarray}
d{\bf v}_{i+1}(t) && = {\bf \tilde{a}}({\bf v}(t^n)) dt + A({\bf v}(t^n)) {\bf v}_{i+1} dt \nonumber \\
&& + B({\bf v}_i)  d{\bf W}(t) ,
\end{eqnarray}
where we have ${\bf v}_i = (v_i, \mu_i, \phi_i)^t$ as the solution vector 
in the $i$-th version, ${\bf \tilde{a}}$ is the vector and $A(t^n)$ is
the Jacobian matrix coming from the linearization,
and $d{\bf W}(t) = (dW_v(t), d W_{\mu}(t), d W_{\phi}(t))^t$ is a 
$3$-dimensional Wiener-process.
We apply the linearization of the convective part, where the matrices are given as:
\begin{eqnarray}
 {\bf a}({\bf v}) & = & {\bf a}({\bf v}(t^n)) + J({\bf v})|_{t^n} ({\bf v} - {\bf v}(t^n)) , \\
& = & \bigg( {\bf a}({\bf v}(t^n)) - J({\bf v})|_{t^n} {\bf v}(t^n) \bigg)  + J({\bf v})|_{t^n} {\bf v} , \\
& = & \tilde{{\bf a}}({\bf v}(t^n))  + J({\bf v})|_{t^n} {\bf v} .
\end{eqnarray}
The Jacobian matrix is given as:
\begin{eqnarray}
\label{num_7}
 J({\bf v})=\begin{bmatrix} \dfrac{\partial a_1}{\partial v} & \dfrac{\partial a_1}{\partial \mu} & \dfrac{\partial a_1}{\partial \phi}  \\ 
\\
\dfrac{\partial a_2}{\partial v} & \dfrac{\partial a_2}{\partial \mu} & \dfrac{\partial a_2}{\partial \phi} \\ 
\\
\dfrac{\partial a_3}{\partial v} &  \dfrac{\partial a_3}{\partial \mu} & \dfrac{\partial a_3}{\partial \phi} \end{bmatrix} =
\begin{bmatrix} \dfrac{\partial F_v(v)}{\partial v} & 0 & 0 \\
\\ 
- 2 \mu \dfrac{\partial D_a(v)}{\partial v} & - 2 D_a(v) & 0 \\ 
\\
0 & 0 & 0 \end{bmatrix} , \\
 J({\bf v})|_{t^n}  =
\begin{bmatrix} \dfrac{\partial F_v(v)}{\partial v}|_{t^n} & 0 & 0 \\
\\ 
- 2 \mu \dfrac{\partial D_a(v)}{\partial v}|_{t^n}  & - 2 D_a(v)|_{t^n} & 0 \\ 
\\
0 & 0 & 0 \end{bmatrix} , 
\end{eqnarray}
\begin{eqnarray}
A({\bf v}(t^n))= J({\bf v})|_{t^n} .
\end{eqnarray}
The fixpoint scheme is given as:
\begin{eqnarray}
\label{iter_1_1}
&& {\bf v}_{i+1}(t^{n+1}) = \exp(A({\bf v}(t^n)) \Delta t) \; \Bigg( {\bf v}(t^n) \nonumber \\
&& + A({\bf v}(t^n))^{-1}  (I - \exp(A({\bf v}(t^n)) \Delta t ))  \; {\bf \tilde{a}}(t^n) \Bigg) \nonumber \\
&& + \int_{t^n}^{\Delta t^{n+1}}  \exp(A({\bf v}(t^n)) (t^{n+1} - s) )  B({\bf v}_i)(s)  d{\bf W}_{\bf v}(s) \Bigg) , 
\end{eqnarray}
We rewrite this with the singular term $A^{-1}$ and obtain:
\begin{eqnarray}
&& {\bf v}_{i+1}(t^{n+1}) = \exp(A({\bf v}(t^n)) \Delta t) \; \Bigg( {\bf v}(t^n) \nonumber \\
&& +  \left(I \Delta t + A({\bf v}(t^n)) \frac{\Delta t^2}{2} + A^2({\bf v}(t^n)) \frac{\Delta t^3}{3!}  \right)  \; {\bf \tilde{a}}(t^n) \Bigg) \nonumber \\
&& + \int_{t^n}^{\Delta t^{n+1}}  \exp(A({\bf v}(t^n)) (t^{n+1} - s) )  B({\bf v}_i)(s)  d{\bf W}_{\bf v}(s) \Bigg) , 
\end{eqnarray}
where $ {\bf \tilde{a}}({\bf v}(t^n)) = \bigg( {\bf a}({\bf v}(t^n)) - A({\bf v}(t^n)) {\bf v}(t^n) \bigg)$
\end{enumerate}

The stochastic integral is computed as a Stratonovich integral,
e.g., Trapezoidal rule:
\begin{eqnarray}
\label{int_1_1}
 && {\bf c}(\Delta t) = \int_{t^n}^{t^{n+1}}  \exp(A({\bf v}(t^{n})) (t^{n+1} - s))  B({\bf v}_i)(s)  dW_s \\
&& = \frac{1}{2} ({\bf W}_{\bf v}(t^{n+1}) - {\bf W}_{\bf v}(t^n))  \bigg(  B({\bf v}_i(t^{n+1}))
\\
&& +  \exp(A({\bf v}(t^{n})) \; \Delta t ) \;  B({\bf v}_i(t^n))  \bigg) , \nonumber \\
&& \Delta t = t^{n+1} - t^n , \\
&& ({\bf W}(t_{j+1}) - {\bf W}(t_j)) = ( rand_1 \sqrt{\Delta t},  rand_2 \sqrt{\Delta t},  rand_3 \sqrt{\Delta t} )^t ,
\end{eqnarray}
where $rand_1$, $rand_2$ and $rand_3$ are three independent random numbers
given with $N(0,1)$.
\end{itemize}

We apply the following errors:
\begin{itemize}
\item The errors are computed as:
\begin{eqnarray}
 err_{v, \Delta t, t=1} = || v_{\Delta t, Scheme}(t=1) - v_{\Delta t_{fine}, Mil}(t=1) || ,
\end{eqnarray}
where $|| \ldots ||$ is the $L_2$-norm, $v_{\Delta t, Scheme}(t=1)$ is the solution
of the applied schemes, \\
which means \\
$Scheme = \{ EM, Mil, Iter1, Iter 2\}$. $\Delta t = \{10^{-4}, 10^{-3}, 10^{-2}, 10^{-1}\}$ are the different time-steps and $t = 1.0$ is the evaluated end-time-point.
$v_{\Delta t_{fine}, Mil}(t=1)$ is a reference solution based on the Milstein-scheme at $t=1.0$ and time-steps $10^{-5}$. \\
The same errors are encountered with the solutions of $\mu$ and $\phi$ \\
(see $err_{\mu, \Delta t, t=1},  err_{\phi, \Delta t, t=1}$).
\item The statistical errors are given as:
\begin{itemize}
\item Strong convergence is based on the errors:
\begin{eqnarray}
err_{v, \Delta t, t=1}, err_{\mu, \Delta t, t=1},  err_{\phi, \Delta t, t=1} .
\end{eqnarray}
\item Weak convergence is based on the mean values of the errors:
\begin{eqnarray}
&& err_{v, \Delta t, t=1, weak} = \frac{1}{N} \sum_{i=1}^N err_{i, v, \Delta t, t=1} , 
\end{eqnarray}
where $err_{i, v, \Delta t, t=1}$ are $i = 1, \ldots, N$ independent errors of the
solution $v$.
\item The derivation of the mean value or variance is given as:
\begin{eqnarray}
 \sigma^2_{v, \Delta t, t=1} = \frac{1}{N-1} \sum_{i=1}^N (err_{i, v, \Delta t, t=1} -  err_{v, \Delta t, t=1, weak})^2 .
\end{eqnarray}

\item Time-averaged mean-square value over the time (scan over the time-space):
\begin{eqnarray}
 \sigma^2_{v, \Delta t} = \frac{1}{T} \sum_{i=1}^N \Delta t \; ( v_{\Delta t, Scheme}(i \; \Delta t) - v_{\Delta t_{fine}, Mil}(i \; \Delta t) )^2 .
\end{eqnarray}
where the time-space is given as $i=1, \ldots, N$, $\Delta t \; N = T = 1$.

\end{itemize}
The same errors and variances are also encountered with the solutions 
of $\mu$ and $\phi$.
\end{itemize}

The solutions of the equations are given for the
different schemes in Figure \ref{coulomb}.
\begin{figure}[ht]
\begin{center}  
\includegraphics[width=5.0cm,angle=-0]{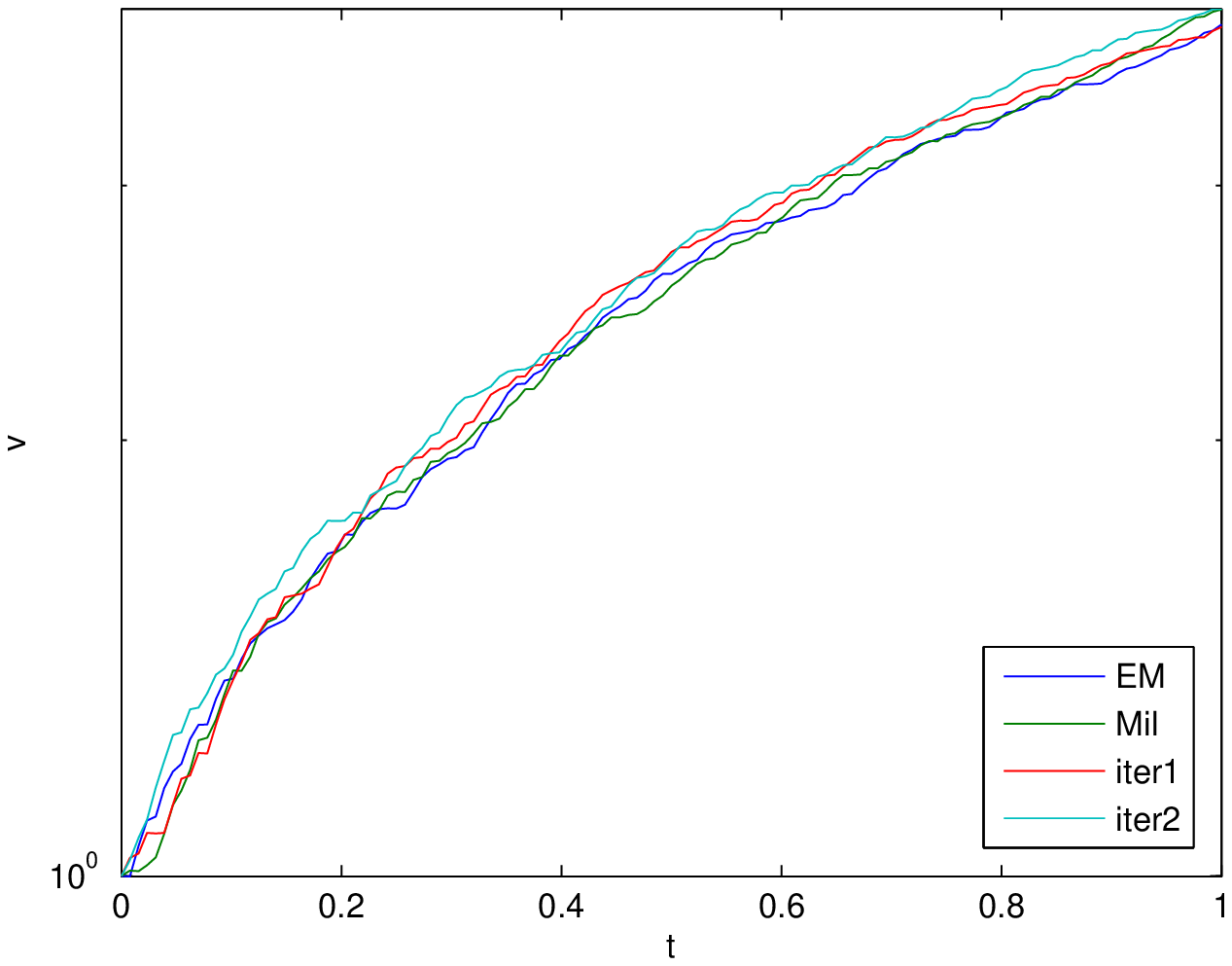} 
\includegraphics[width=5.0cm,angle=-0]{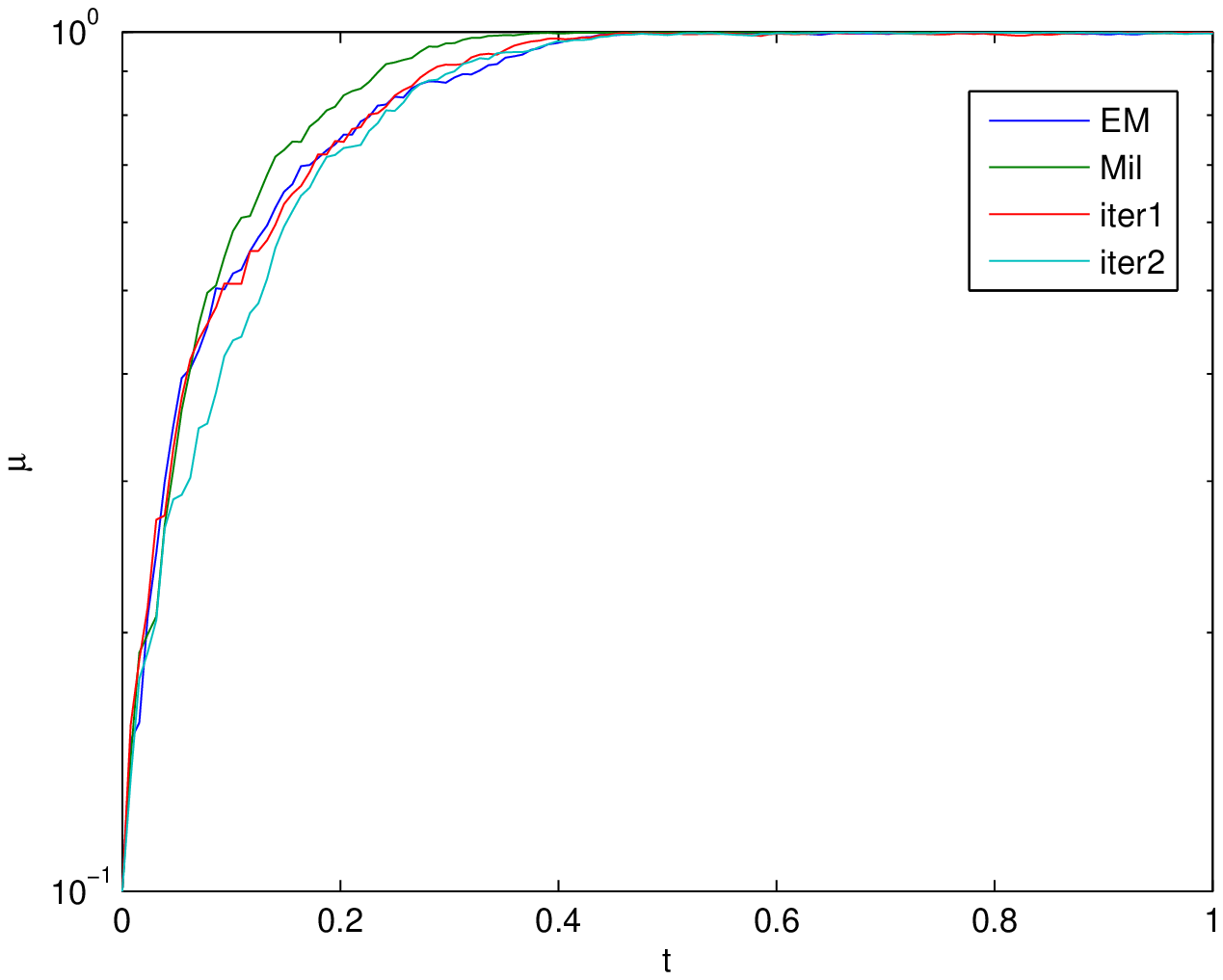}
\includegraphics[width=5.0cm,angle=-0]{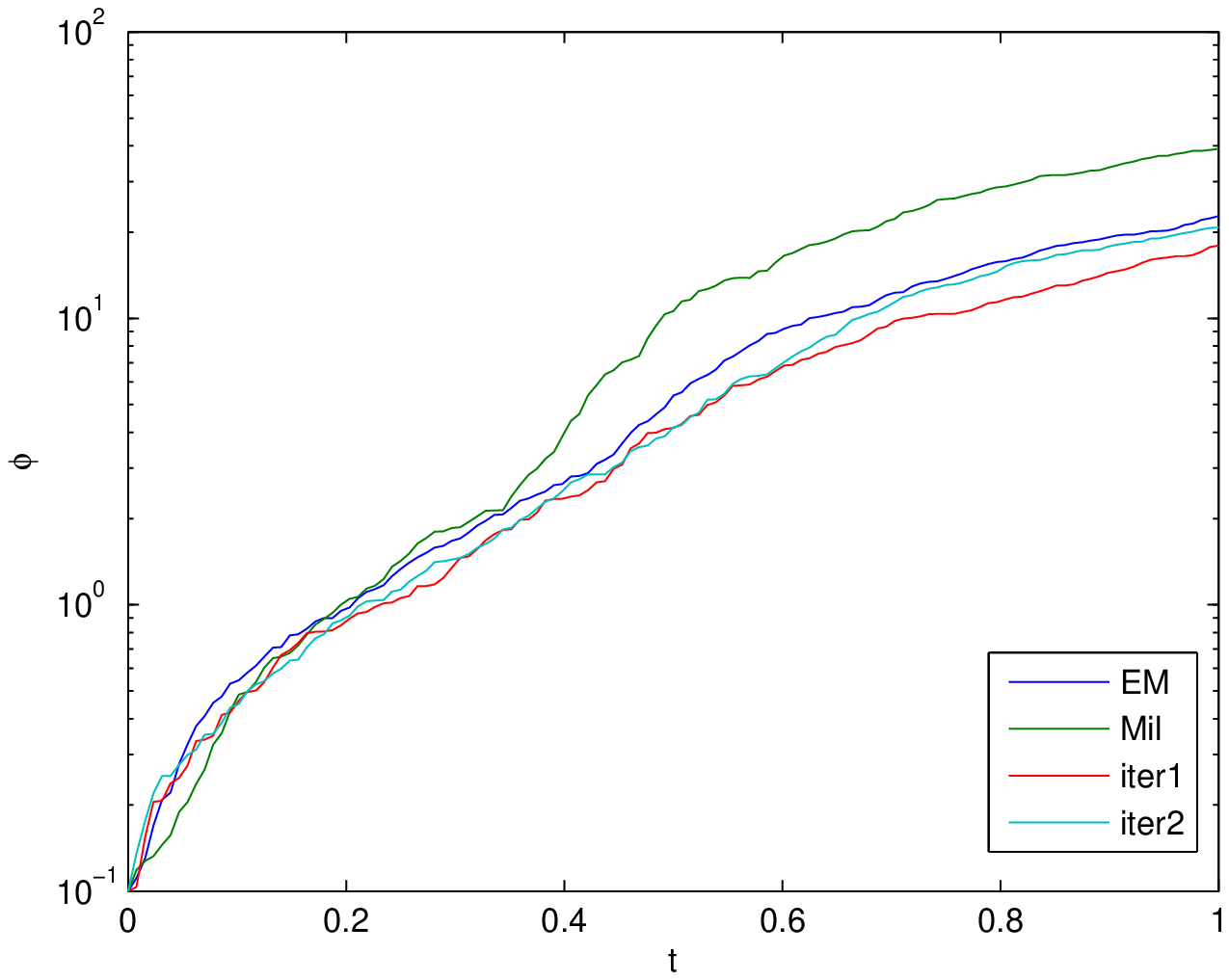} 
\end{center}
\caption{\label{coulomb} The figures present the results of the
different splitting schemes (EM: Euler-Maruyama, Iter1: Splitting Version 1, Iter2: Splitting Version 2. The upper left figure presents the solutions of $v$,
the upper right figure presents the solutions of $\mu$ and the lower figure presents the solution of $\phi$.}
\end{figure}

The convergence results of the different schemes and the three dimensional plots are given in Figure \ref{coulomb_strong}.
\begin{figure}[ht]
\begin{center}  
\includegraphics[width=5.0cm,angle=-0]{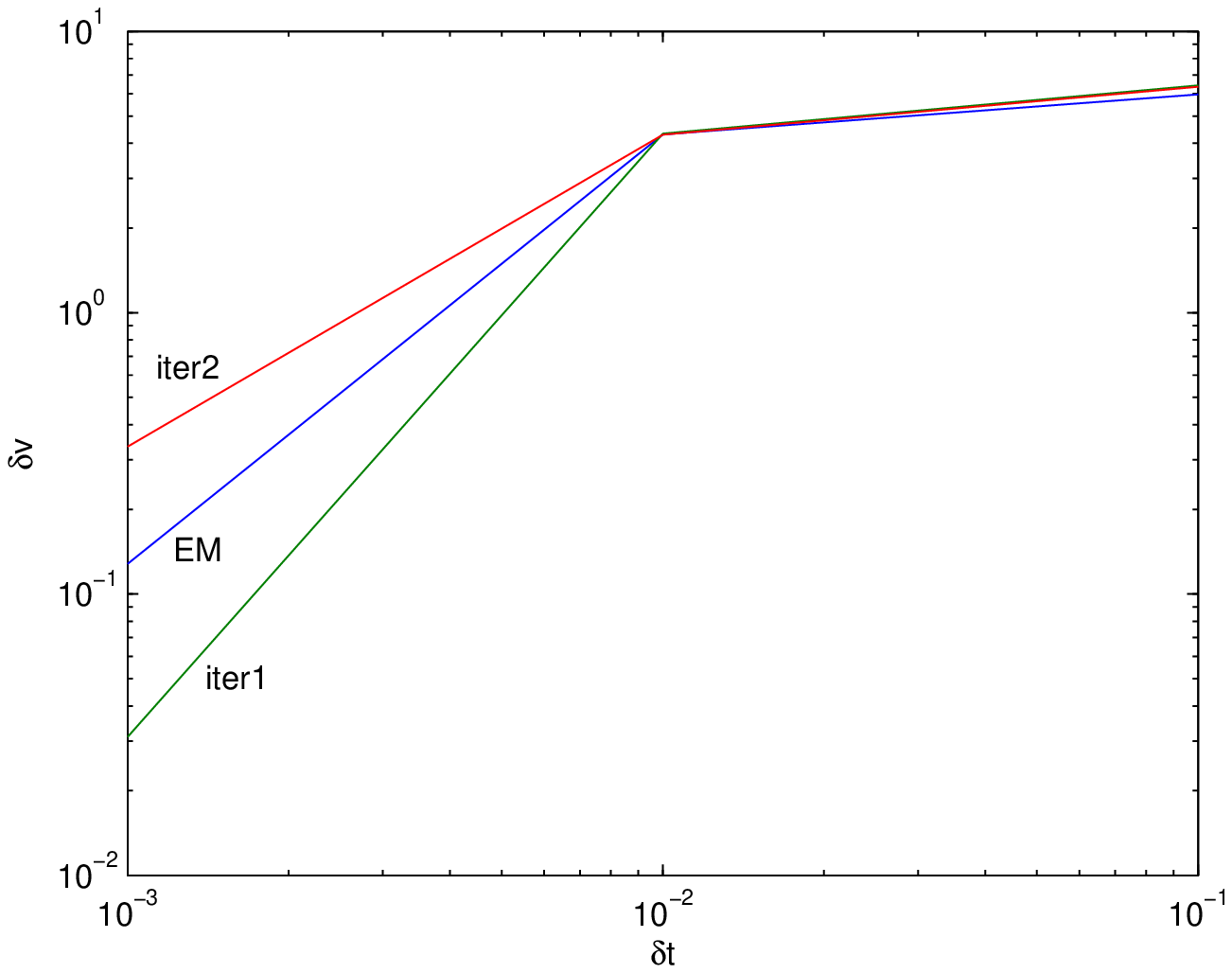} 
\includegraphics[width=5.0cm,angle=-0]{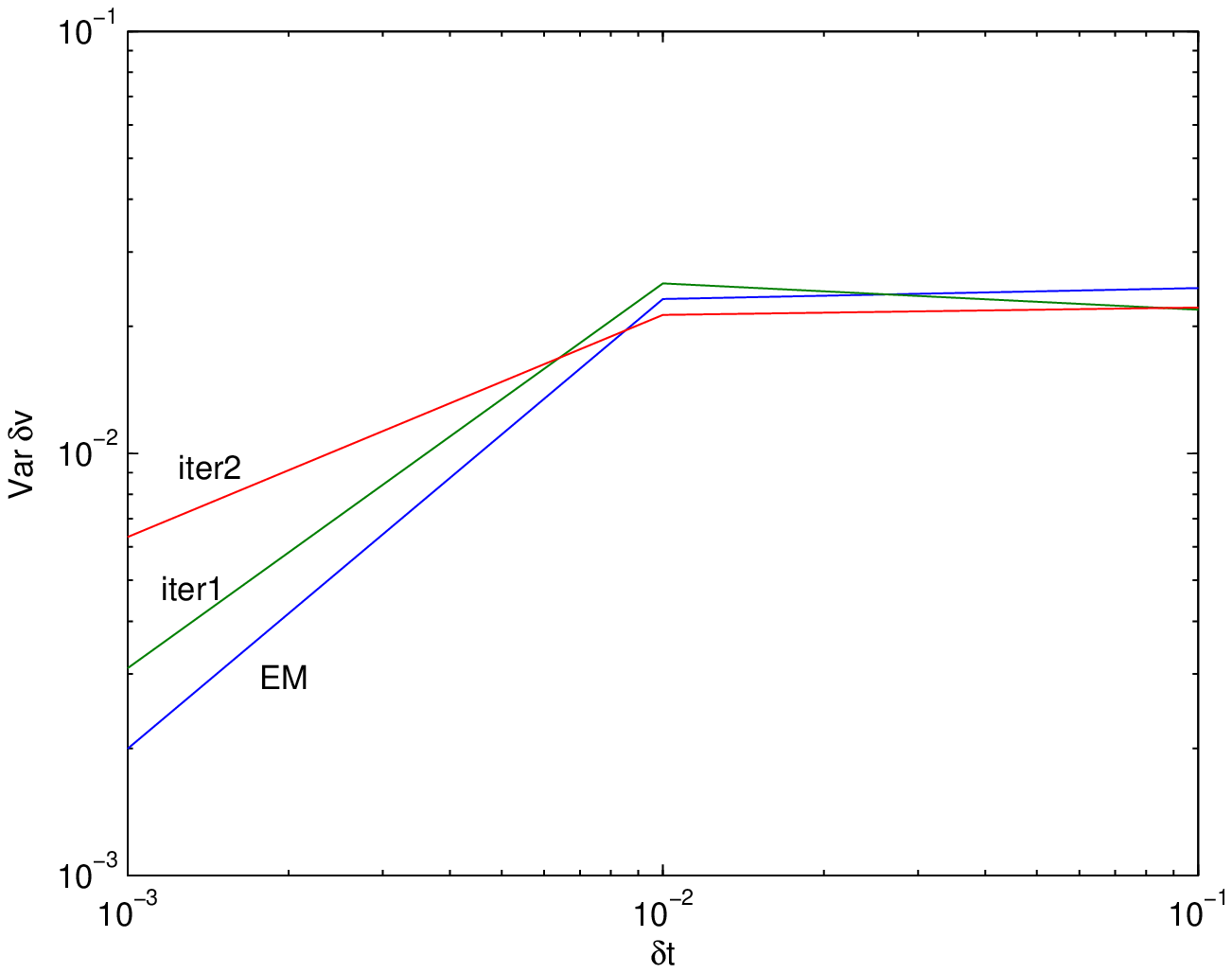} 
\includegraphics[width=5.0cm,angle=-0]{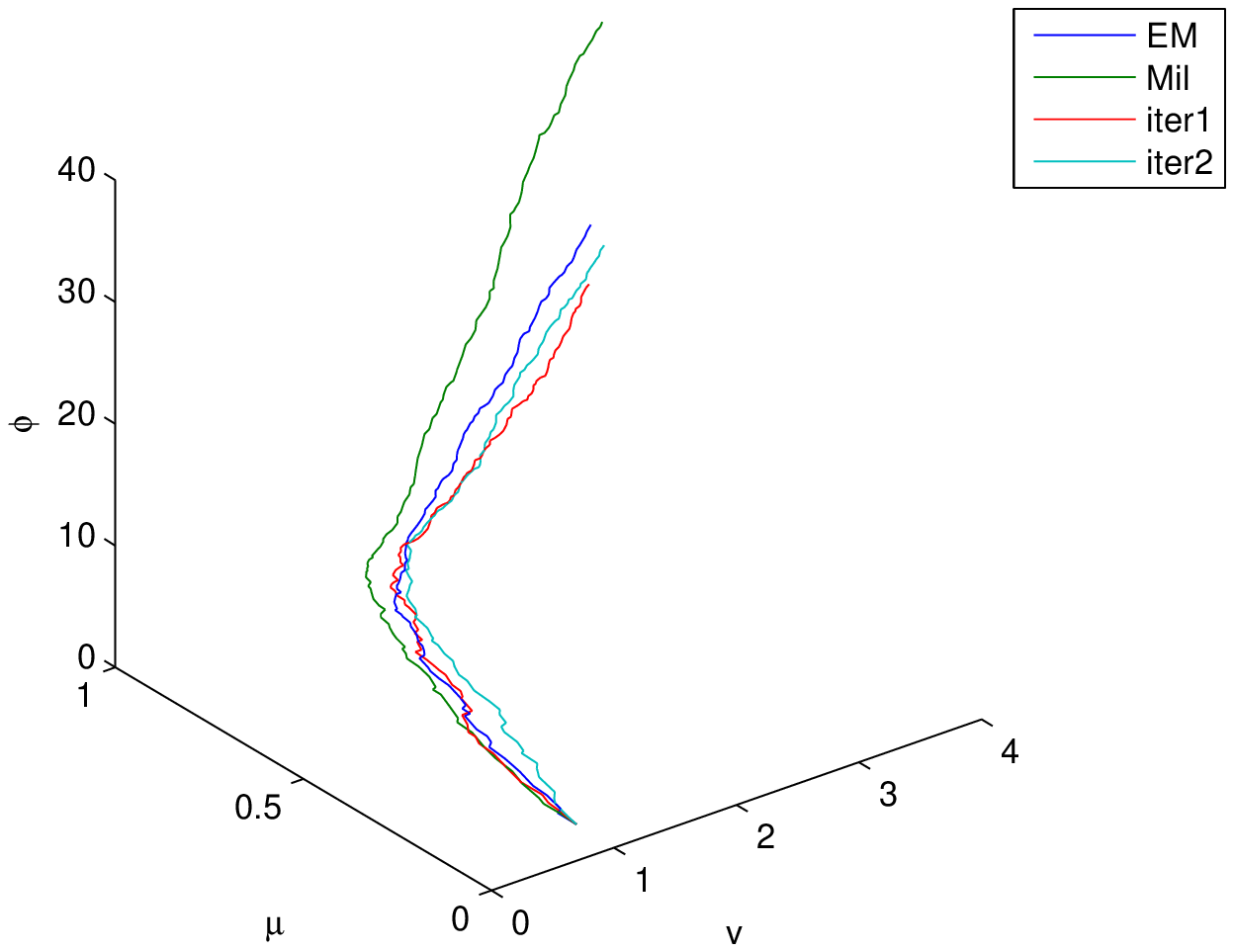} 
\end{center}
\caption{\label{coulomb_strong}  The figures present the results of the
different splitting schemes (EM: Euler-Maruyama, Iter1: Splitting Version 1, Iter2: Splitting Version 2. 
The upper left figure presents the weak convergence of $v$; the upper right figure presents variance of of $v$ and the lower figure presents the three dimensional plot of all the solutions.
}
\end{figure}

%

%

In the following, the computational time of the different schemes
are given (see Table \ref{table_1}). We obtain, that the explicit schemes,
i.e., Euler-Maruyama and Milstein scheme, are faster but they have only their 
restrictions to small time-steps. Therefore, the benefit of the implicit-iterative schemes, i.e., iterative splitting (iter1 and iter2),
is given based on large time-steps, e.g., $\Delta t \le 10^{-1}$, where the
explicit scheme are oscillating. 

\begin{table}[h]
\begin{center}
\begin{tabular}{||c||c | c| c| c||}
\hline
Method  &        \multicolumn{4}{c||}{$\Delta t$ } \\	
    & $10^{-1}$ & $10^{-2}$ & $10^{-3}$ & $10^{-4}$ \\
\hline
Euler-Maruyama & 7.7248e-04s &	0.0018s & 0.0132s & 0.1517s \\
Milstein &	0.0012s	&	0.0032s	& 0.0286s &0.3215s \\
iter1	 &	0.0080s	&	0.0536s	& 0.5302s & 6.1977s \\
iter2	 &	0.0078s	&	0.0472s	& 0.4497s & 5.1896s \\
\hline
\end{tabular}
\caption{\label{table_1} Computational time of the different solver methods.}
\end{center}
\end{table}

\begin{remark}
The examples show the important selections of the linearization method, which is related to the iterative schemes. Some small benefits are obtained with the
Version 1, see Equation (\ref{iter_1_1}), that applied a simple relaxation of the nonlinear part. Here, we take into account the relaxation effect of the 
iterative schemes as a function of time step. 
We see an improvement with larger time-steps, e.g. see the variance-errors in Figure \ref{coulomb_strong}.
On the other hand, we have taken into account the costs of the new algorithms,
that are acceptable, e.g., 2-3 times that of the standard schemes.
\end{remark}

\section{Conclusion}
\label{conclu}

We discuss the problems of using novel iterative splitting schemes
to solve stochastic differential equations, which are applied to 
Langevin equations.
We derive convergence results to the iterative schemes
and see the benefit of higher order reconstruction based on the number
of iterative steps.
The numerical examples present the advantages of the
iterative schemes and their computational costs with respect to their 
relaxation effects.
A real-life problem based on a collision model is presented.
The novel schemes can be applied to nonlinear problems
and they allow to use larger time steps without loosing their numerical
accuracy. Here, we can optimize the application of such novel schemes, 
while the computational costs for the standard schemes are higher with 
smaller time steps. In future, we see an area to optimize such novel schemes with their benefit of relaxing the nonlinear solutions and to apply larger time steps.

\bibliographystyle{plain}

\end{document}